\documentclass[preprint,10pt]{elsarticle}

\usepackage{amssymb}
\usepackage{amsmath}
\usepackage[all,cmtip]{xy}
\usepackage{latexsym}
\usepackage{mathtools}
\usepackage{amsthm}
\usepackage{a4wide}
\usepackage{color}
\usepackage{hyperref}
\usepackage{tikz}
\usetikzlibrary{arrows}
\input diagxy

\journal{}

\theoremstyle{plain}
  \newtheorem{thm}{Theorem}[section]
  \newtheorem{lem}[thm]{Lemma}
  \newtheorem{prop}[thm]{Proposition}
  \newtheorem{cor}[thm]{Corollary}
\theoremstyle{definition}
  \newtheorem{defn}[thm]{Definition}
  
  \newtheorem{exmp}[thm]{Example}
  \newtheorem{rem}[thm]{Remark}

\def\rto{{\bfig\morphism<180,0>[\mkern-4mu`\mkern-4mu;]\place(82,0)[\mapstochar]\efig}}

\makeatletter
\def\ps@pprintTitle{%
 \let\@oddhead\@empty
 \let\@evenhead\@empty
 \def\@oddfoot{\centerline{\thepage}}%
 \let\@evenfoot\@oddfoot}
\makeatother

\newcommand{\lda}{\swarrow}
\newcommand{\rda}{\searrow}

\newcommand{\Lra}{\Longrightarrow}

\newcommand{\bv}{\bigvee}
\newcommand{\bw}{\bigwedge}

\newcommand{\dv}{\dashv}

\newcommand{\lam}{\lambda}

\newcommand{\si}{\sigma}

\renewcommand{\phi}{\varphi}
\newcommand{\QRel}{\CQ\text{-}\Rel}
\newcommand{\phila}{\overleftarrow{\phi}}

\newcommand{\CC}{\mathcal{C}}

\newcommand{\CP}{\mathcal{P}}
\newcommand{\CQ}{\mathcal{Q}}
\newcommand{\CR}{\mathcal{R}}

\newcommand{\Fx}{\mathfrak{x}}
\newcommand{\Fy}{\mathfrak{y}}
\newcommand{\Fz}{\mathfrak{z}}
\newcommand{\Fa}{\mathfrak{a}}
\newcommand{\Fb}{\mathfrak{b}}
\newcommand{\Fw}{\mathfrak{w}}

\newcommand{\FZ}{\mathfrak{Z}}

\newcommand{\sD}{{\sf D}}

\newcommand{\sP}{{\sf P}}

\newcommand{\sV}{{\sf V}}

\newcommand{\sfs}{{\sf s}}
\newcommand{\sy}{{\sf y}}

\newcommand{\bbP}{\mathbb{P}}
\newcommand{\bbT}{\mathbb{T}}

\newcommand{\Cat}{{\bf Cat}}

\newcommand{\Rel}{{\bf Rel}}

\newcommand{\Set}{{\bf Set}}

\newcommand{\VCat}{{\sf V}\text{-}\Cat}

\newcommand{\QCat}{\CQ\text{-}\Cat}

\newcommand{\op}{{\rm op}}

\newcommand{\hT}{\hat{T}}
\begin{document}

\begin{frontmatter}



\title{Lax Distributive Laws for Topology, I}


\author{Walter Tholen\fnref{A}}
\ead{tholen@mathstat.yorku.ca}

\address{Department of Mathematics and Statistics, York University, Toronto, Ontario, Canada, M3J 1P3}
\address{\rm Dedicated to the memory of Bob Walters}

\fntext[A]{Partial financial assistance by the Natural Sciences and Engineering Research Council (NSERC) of Canada is gratefully acknowledged.}

\begin{abstract}
For a quantaloid $\CQ$, considered as a bicategory, Walters introduced categories enriched in $\CQ$. Here we extend the study of monad-quantale-enriched categories of the past fifteen years by introducing monad-quantaloid-enriched categories. We do so by making lax distributive laws of a monad $\mathbb T$ over the discrete presheaf monad of the small quantaloid $\CQ$ the primary data of the theory, rather than the lax monad extensions of $\mathbb T$ to the category of $\CQ$-relations that they equivalently describe. The central piece of the paper establishes a Galois correspondence between such lax distributive laws and lax Eilenberg-Moore ${\mathbb T}$-algebra structures on the set of discrete presheaves over the object set of $\CQ$. We give a precise comparison of these structures with the more restrictive notion introduced by Hofmann in the case of a commutative quantale, called natural topological theories here, and describe the lax monad extensions introduced by him as minimal. Throughout the paper, a variety of old and new examples of ordered, metric and topological structures illustrate the theory developed, which includes the consideration of algebraic functors and change-of-base functors in full generality.
\end{abstract}


\end{frontmatter}


\section{Introduction}
For monads $\mathbb S$ and $\mathbb T$ on a category $\CC$, liftings of $\mathbb S$ along the forgetful functor
$\CC^{\mathbb T}\to\CC$ of the Eilenberg-Moore category of $\mathbb T$, or extensions of $\mathbb T$ along the insertion functor $\CC\to\CC_{\mathbb S}$ to the Kleisli category of $\mathbb S$, correspond precisely to Beck's \cite{Beck} {\em distributive laws} $\lam:TS\to ST$ of $\mathbb T$ over $\mathbb S$; see \cite{BarrWells} and II.3 of \cite{MonTop} for a compact account of these correspondences. For $\CC={\bf Set}, {\mathbb T}={\mathbb L}$ the free monoid (or {\em list}) monad, and $\mathbb S$ the free Abelian group monad, their algebraic prototype interpretes the left-hand terms of the equations
$$x(y+z)=xy+xz\quad\text{      and      }\quad(x+y)z=xz+yz$$
as elements of the free monoid $LSX$ over (the underlying set of) the free Abelian group $SX$ over some alphabet $X$ and assigns to them the right-hand terms in $SLX$, to then obtain the category of (unital) rings as the Eilenberg-Moore algebras of a composite monad
$\mathbb{SL}$ as facilitated by $\lam$. Similarly, keeping ${\mathbb T}={\mathbb L}$ but letting now ${\mathbb S}={\mathbb P}$ be the power set monad, the distributive law
$$\lam_X:L\sP X\to\sP LX,\quad(A_1,...,A_n)\mapsto A_1\times...\times A_n,$$
produces a composite monad whose Eilenberg-Moore category is the category of {\em quantales}, {\em i.e.}, of the monoid objects in the monoidal-closed category {\bf Sup} of sup-lattices (see \cite{JoyalTierney, Rosenthal}), characterized as the complete lattices with a monoid structure whose multiplication distributes over arbitrary suprema in each variable.
Ever since the appearance of Beck's original work, distributive laws have been, and continue to be, studied from a predominantly algebraic perspective, at many levels of generality; see, for example, \cite{Street, LackStreet, Hermida, Bohm}. But what is their role in topology, if any?

As a unification of the settings used by Lawvere \cite{Lawvere} and by Manes \cite{Manes} and Barr \cite{Barr} for their respective descriptions of metric spaces and topological spaces, the viewpoint of {\em Monoidal Topology} \cite{CleHof2003, ClementinoTholen2003, CleHofTho, Seal, Hofmann, MonTop} has been that some key categories of analysis and topology are described as categories of lax
 $({\mathbb T},\sV)$-algebras, also called $({\mathbb T},\sV)$-{\em categories}, where $\sV$ is a quantale and $\mathbb T$ a $\Set$-monad with a lax extension to the category $\sV\text{-}{\bf Rel}$ of sets and $\sV$-valued relations (or {\em matrices} \cite{BCSW}) as morphisms.
 For example,  for $\sV={\sf 2}$ the two-element chain and for ${\mathbb T}={\mathbb U}$ the ultrafilter monad with its lax Barr extension to relations, one obtains the Manes-Barr presentation of topological spaces in terms of ultrafilter
 convergence (with just two axioms that generalize reflexivity and transitivity of ordered sets). With the same monad, but now with $\sV=[0,\infty]$ being Lawvere's extended real half-line and addition playing the role of the tensor product, one obtains Lowen's \cite{Lowen} category of {\em approach spaces}, which incorporates both Barr's {\bf Top} and Lawvere's {\bf Met} in a satisfactory manner. Perhaps one of the best successes of the subject so far has been the strictly equational characterization of exponential objects in the lax setting of $({\mathbb T},\sV)$-categories. For the extensive literature on the subject, we must refer the reader to the literature list in \cite{MonTop}, in particular the Notes to Chapters III and IV of \cite{MonTop}, which also list many
 important related approaches, such as that of Burroni \cite{Burroni} (which drew Lambek's \cite{Lambek} multicategories into the setting) and the thesis of M\"{o}bus \cite{Mobus} (which, beyond compactness and Hausdorff separation, explored a wide range of topological concepts in the relational monadic setting).

 For the general $({\mathbb T},\sV)$-setting, it  had been realized early on that $\sV\text{-}{\bf Rel}$ is precisely the Kleisli category of the $\sV$-power set monad
  ${\mathbb P}_{\sV}$ (with $\sP_{\sV}X=\sV^X$), and it was therefore plausible that {\em lax} extensions $\hat{T}$ of $\mathbb T$ to $\sV\text{-}{\bf Rel}$
  correspond to  {\em monotone lax} distributive laws of $\mathbb T$ over ${\mathbb P}_{\sV}$
  (see \cite{Schubert} and Exercise III.1.I of \cite{MonTop}). Our initial point in this paper is to underline the role of lax distributive laws as the {\em primary} data in the study of topological categories, rather than as some secondary data derived from lax monad extensions, the establishment of which can be tedious (see \cite{CleHof2003,Seal}). In fact, in analyzing
step by step the correspondence between the two entities (as we will do in Section 6 of this paper), we see that lax distributive laws minimize the number of variables in, and often the computational effort for, checking the required inequalities. It is therefore consequential that here we express
$({\mathbb T},\sV)$-categories directly as {\em lax} $\lam$-{\em algebras}, without prior reference to the lax monad
extension which the ambient lax distributive law $\lam$ corresponds to. Thus, their axioms are entirely expressed in terms of maps, rather than $\sV$-relations, and of the two $\Set$-monads at play, $\mathbb T$ and ${\mathbb P}_{\sV}$. We note that, to date, the strict counterpart of the notion of lax $\lam$-algebra as introduced in Section 4 does not seem to have been explored yet -- and may be of much lesser importance than the lax version --, but must in any case not be confused with a different notion appearing in IV.3 of Manes' book \cite{Manes1976}.

In fact, in this paper we present the lax distributive laws and their equivalent lax monad extensions, together with their isomorphic model categories ({\em i.e.}, lax $\lam$-algebras {\em vs.} $({\mathbb T},\sV)$-categories) at a considerably generalized level, by replacing the quantale $\sV$ by a small {\em quantaloid} $\CQ$,
{\em i.e.}, by a small category (rather than a monoid) enriched in the category {\bf Sup} of complete lattices and their suprema preserving maps (see \cite{RosenthalQuantaloids, Stubbe2005, Stubbe2006, Heymans}). For this to work, $\mathbb T$ must now be a monad on the comma category $\Set/\CQ_0$, with $\CQ_0$ the set of objects of $\CQ$, rather than just on $\Set$ as in the quantale case when $\CQ_0\cong 1$ is a singleton set.
However, noting that every $\Set$-monad $\mathbb T$ lifts to a $\Set/\CQ_0$-monad when $\CQ_0$ carries a Eilenberg-Moore $\mathbb T$-algebra structure, one realizes immediately that the range of applications is not at all reduced by moving to the comma category. The opposite is true, even when $\mathbb T$ is the identity monad and $\lam$ the identity transformation of the discrete presheaf monad $\sP_{\CQ}$, where lax $\lam$-algebras are simply $\CQ$-categories, as first considered in Walters' pioneering note \cite{Walters}. More generally then, in the hierarchy
$$\bfig
\Atriangle/-`-`/<700,400>[\text{closed
bicategories}`\text{quantaloids}`\text{monoidal-closed categories};``]
\Vtriangle(0,-400)/`-`-/<700,400>[\text{quantaloids}`\text{monoidal-closed
categories}`\text{quantales};``]
\efig$$
we add a monad to the enrichment through quantaloids, thus complementing the corresponding past efforts for quantales and monoidal-closed categories, and leaving the field open for future work on closed bicategories. In doing so, our focus is not on a generalization {\em{per se}}, but rather on the expansion of the range of meaningful examples. In fact, through the consideration of quantaloids that arise from quantales via the well-studied Freyd-Grandis ``diagonal construction", originating with \cite{Freyd}, presented in \cite{Grandis}, and used by many authors
(see, for example, \cite{HohleKubiak, PuZhang, Stubbe2014}), we demonstrate that the quantaloidic context allows for the incorporation of many ``partially defined" structures, which typically relax the reflexivity condition of the total context in a meaningful way.

In \cite{Hofmann}, Hofmann gave the notion of a (lax) {\em topological theory} which, in the presence of the $\Set$-monad $\mathbb T$ and the commutative quantale $\sV$, concentrates all needed information about the specific
Barr-type lax extension of $T$ to $\sV\text{-}{\bf Rel}$ into a (lax) $\mathbb T$-algebra structure $\xi:T\sV\to\sV$ on the set $\sV$, such that $\xi$ makes the monoid operations $\otimes: \sV\times\sV\to\sV$
 and ${\sf k}:1\to\sV$ lax $\mathbb T$-homomorphisms and satisfies a monotonicity and naturality condition. While in \cite{ClementinoTholen} we characterized the Barr-Hofmann lax extensions of $\mathbb T$ arising from such theories among all lax extensions, the two main results of this paper clarify the role of Hofmann's notion in the quantale setting and extend it considerably to the more general context of a quantaloid $\CQ$. First, in Section 5 we establish a Galois correspondence between monotone lax distributive laws of a given monad
 $\mathbb T$ on $\Set/\CQ_0$ and certain lax $\mathbb T$-algebra structures $\xi$ on $\sP_{\CQ}\CQ_0$. The lax distributive laws closed under this correspondence, called {\em maximal}, give rise to new types of lax monad extension that don't seem to have been explored earlier. Secondly, in Theorem \ref{comparison theorem}, we give a precise comparison of our notion of topological theory (as given in Definition \ref{toptheory}) with Hofmann's more restrictive notion. We also give a context in which the Hofmann-type extensions are characterized as {\em minimal} (see Theorem \ref{minimal}). Let us emphasize that the conditions on the cartesian
 binary and nullary monoid operations used by Hofmann don't compare easily with the conditions on the multiplication and unit of the discrete presheaf monad as used in our setting, and they don't seem to be amenable to direct extension from the context of a commutative quantale to that of a quantaloid.
 For an overview chart on the relationships between lax distributive laws, lax monad extensions, and topological theories, we refer to Section 8.

 A brief outlook on the forthcoming paper \cite{LaiShenTholen} seems to be in order, where we present the non-discrete counterpart of the theory presented here, thus considering monads on the category $\CQ\text{-}\Cat$ of small $\CQ$-categories and their lax distributive laws over the (full) presheaf monad. It is clear from the outset, and largely verified by the existing works on
 monad-quantale-enriched categories, that this setting will make for a more satisfactory theory, simply because the full presheaf monad, unlike its discrete counterpart, is lax idempotent (or {\em of Kock-Z\"{o}berlein type}). Nevertheless, the prior consideration of the discrete case in this paper seems to be a necessary step, in order for us to be able to provide a viable array of monads on $\CQ\text{-}{\Cat}$
 since, with a lax extension of a monad on $\Set/\CQ_0$
  at hand, it is easy to ``lift" monads on $\Set/\CQ_0$ to $\CQ\text{-}\Cat$ (as has been done in \cite{TholenOrd} in the case of a quantale).

For general categorical background, we refer the reader to \cite{MacLane, Adamek1990, Borceux, Kelly1982}.

{\em Acknowledgement.} Parts of the theory developed in the paper have been presented in talks at the Joint Meeting of the American and Portuguese Mathematical Societies in Oporto (Portugal) in June 2015 and at Sichuan and Nanjing Universities in November 2015. I am grateful for helpful comments received, especially from Dirk Hofmann, Maria Manuel Clementino, Gavin Seal, Lili Shen, Hongliang Lai, and Dexue Zhang.

\section{Quantaloid-enriched categories}
A {\em quantaloid} is a category $\CQ$ enriched in the monoidal-closed category {\bf Sup} \cite{JoyalTierney} of complete lattices with suprema-preserving maps; hence, the hom-sets of $\CQ$ are complete lattices, and composition of morphisms from either side preserves arbitrary suprema and has therefore right adjoints. As a consequence, one has binary operations $\rda$ and $\lda$ representing the ``internal homs", that is: for $u:r\to s, \;v:s\to t, \;w:r\to t$ in $\CQ$ one has the morphisms
$(v\rda w):r\to s, \;(w\lda u):s\to t$ given by the equivalences

$$u\leq v\rda w \iff v\circ u\leq w \iff v\leq w\lda u.$$
A {\em lax homomorphism} $\varphi :\CQ\to \CR$ of quantaloids is a lax functor (thus satisfying the rules $1_{\varphi t}\leq \varphi 1_t$ and $\varphi v \circ \varphi u \leq \varphi(v\circ u)$) which
maps hom-sets monotonely; $\varphi$ is a (strict) {\em homomorphism} if $\varphi$ is a functor preserving suprema taken in the hom-sets.
We denote the category of small quantaloids and their (lax) homomorphisms by {\bf Qnd} ({\bf LaxQnd}). The set-of-objects functor

$$(-)_0:{\bf LaxQnd}\to {\bf Set}, \CQ\mapsto {\rm ob}\CQ =: \CQ_0$$
has a right adjoint $(-)_c$ which provides each set X with the chaotic order and considers it as a category $X_c$ with $(X_c)_0=X$, so that for all $x,y\in X$ there is exactly one morphism $x\to y$, denoted by $(x,y)$; having singleton hom-sets only, $X_c$ is trivially a quantaloid, and every {\bf Set}-map becomes a homomorphism.

Throughout the paper, let $\CQ$ be a small quantaloid. A small $\CQ$-{\em category} is a set $X$ provided with a lax homomorphism $a:X_c\to \CQ$. Its object part $a: X\to \CQ_0$ assigns to every $x\in X$ its {\em array} (also called {\em type} or {\em extent}) $ax\in \CQ_0$, often denoted by $|x|=|x|_X=ax$, and its morphism part gives for all $x,y\in \!X\text{  }\CQ$-morphisms $a(x,y):|x|\to |y|$, subject to the rules

$$1_{|x|}\leq a(x,x),\quad a(y,z)\circ a(x,y)\leq a(x,z).$$
A $\CQ$-{\em functor} $f:(X,a)\to (Y,b)$ is an array-preserving map $f:X\to Y$ with $a(x,y)\leq b(fx,fy)$ for all $x,y \in X$. In other words then, the resulting category $\CQ$-{\bf Cat} of small $\CQ$-categories and their $\CQ$-functors is the lax comma category of small chaotic quantaloids over $\CQ$, and one has the set-of-objects functor

$$\bfig
\morphism<1200,0>[\QCat`\Set/\CQ_0;(-)_0]
\Vtriangle(-300,-500)<300,300>[X_c`Y_c`\CQ;f`a`b]
\Vtriangle(900,-500)<300,300>[X`Y`\CQ_0;f`|\text{-}|_X`|\text{-}|_Y]
\place(600,-350)[\mapsto] \place(0,-320)[\leq]
\efig$$
to the comma category of sets over $\CQ_0$. In what follows, we will often write $X$ instead of $X_c$ or $(X,a)$.

An easily proved (see \cite{ShenTholen2015}), but useful, fact is:

\begin{prop}\label{topologicity1}
The functor $(-)_0$ is topological (in the sense of {\rm{\cite{Herrlich1974}}}) and, as a consequence, $\CQ$-{\bf Cat} is totally complete and totally cocomplete (in the sense of {\rm{\cite{StreetWalters}}}).
\end{prop}
\begin{proof}
The $(-)_0$-initial structure $a$ on $X$ with respect to a family of $(\bf{Set}/\CQ_0)$-morphisms $f_i:X\to Y_i $ with each $Y_i$ carrying the $\CQ$-category structure $b_i (i\in I)$ is given by
$$a(x,y)=\bw_{i\in I}b_i(f_ix,f_iy),$$
with $x,y \in X$.
 \end{proof}
Incidentally, it seems fitting to note here that topologicity of a faithful functor is characterized as total cocompleteness when the concrete category in question is considered as a category enriched over a certain quantaloid: see \cite{Garner, ShenTholen2016}.

Next, one easily sees that every lax homomorphism  $\varphi: \CQ\to \CR$ of quantaloids induces the {\em change-of-base functor}
$$B_{\varphi}:\CQ\text{-\bf{Cat}}\to \CR\text{-\bf{Cat}},\;(X,a)\mapsto (X,\varphi a),$$
which commutes with the underlying {\bf Set}-functors. More precisely, with $B_{\phi_0}$ denoting the effect of $B_{\phi}$ on the underlying sets over $\CQ_0$, one has the commutative diagram of functors which exhibits $(B_{\varphi},B_{\phi_0})$ as a morphism of topological functors:
$$\bfig
\square<1000,400>[\QCat`\CR\text{-}\bf{Cat}`\Set/\CQ_0`\Set/\CR_0;B_{\varphi}`(-)_0`(-)_0`B_{\varphi_0}]
\efig$$
Obviously, $B_{\varphi}$ preserves $(-)_0$-initiality when $\varphi$ preserves infima. Let us also mention that, if we order the hom-sets of {\bf LaxQnd} by
$$\varphi \leq \psi \iff \forall u: r\to s \text{ in } \CQ: \varphi r=\psi r, \varphi s=\psi s \text{ and } \varphi u\leq\psi u,$$
then $\varphi \leq \psi$ gives a natural transformation $B_{\varphi}\to B_{\psi}$ whose components at the {\bf Set}-level are identity maps; thus a 2-functor $B_{(-)}:{\bf LaxQnd} \to {\bf CAT}$ emerges.

The one-object quantaloids are the (unital) {\em quantales}, {\em i.e.}, the complete lattices $\sf V$ that come with a monoid structure whose binary operation $\otimes$ preserves suprema in each variable. We generally denote the $\otimes$-neutral element by $\sf k$;
so, in quantaloidic terms, ${\sf k}=1_*$, when we denote by $*$ the only object of $\sf V$ as a category. Let us record here a well-known list of relevant quantales $\sf V$ with their induced categories $\VCat$.

\begin{exmp}\label{firstexamples}
\begin{itemize}
\item[\rm{(1)}] The terminal quantaloid {\sf 1} is a quantale, and ${\sf 1}\text{-}{\bf Cat}={\bf Set}$. The initial quantale is (as a lattice) the two-element chain ${\sf 2}=\{\perp<\top\}$, with $\otimes=\wedge, {\sf k}=\top$, and ${\sf 2}\text{-}{\bf Cat}$ is the category {\bf Ord} of preordered sets and monotone maps. (In what follows, we suppress the prefix ``pre" in ``preorder", adding ``separated" whenever antisymmetry is required.)
\item[\rm{(2)}] $[0,\infty]$ denotes the extended real line, ordered by the natural $\geq$ (so that 0 becomes the largest and $\infty$ the least element) and considered as a quantale with the binary operation +, naturally extended to $\infty$. (This is the monoidal-closed category first considered by Lawvere \cite{Lawvere}.) A $[0,\infty]$-category is a generalized metric space, {\em i.e.,} a set $X$ provided with a function $a:X\times X\to [0,\infty]$ with $a(x,x)=0$ and $a(x,z)\leq a(x,y)+a(y,z)$ for all $x,y,z \in X$;
$[0,\infty]$-functors are non-expanding maps. We write ${\bf Met} =[0,\infty]$-{\bf Cat} for the resulting category. The only homomorphism $2\to [0,\infty]$ of quantales has both a left and a right adjoint, hence there is an embedding ${\bf Ord}\to {\bf Met}$ that is both reflective and coreflective.
\item[\rm{(3)}] The quantale $[0,\infty]$ is of course isomorphic to the unit interval $[0,1]$, ordered by the natural $\leq$ and provided with the multiplication. Interpreting $a(x,y)\in [0,1]$ as the probability that $x,y\in X$ be related under a given random order $\tilde{a}$ on $X$, we call
$(X,a)\in [0,1]\text{-}{\bf Cat}$ a
{\em probabilistic ordered set} and denote the resulting cateory by {\bf ProbOrd}, which, of course, is just an isomorphic guise of {\bf Met}.

Both, $[0,\infty]$ and $[0,1]$ are embeddable into the quantale $\bf {\Delta}$ of all {\em distance distribution functions} $\varphi: [0,\infty]\to[0,1]$, required to satisfy the left-continuity condition $\phi(\beta)={\rm sup}_{\alpha<\beta}\phi(\alpha)$, for all $\beta\in [0,\infty]$. Its order is inherited from $[0,1]$, and its monoid structure is given by the commutative
{\em convolution} product $(\phi\otimes\psi)(\gamma)={\rm sup}_{\alpha+\beta\leq\gamma}\phi(\alpha)\psi(\beta)$;
the $\otimes$-neutral function $\kappa$ satisfies $\kappa(0)=0$ and $\kappa(\alpha)=1$ for all $\alpha >0$.
Interpreting $a(x,y)(\alpha)$ as the probability that a given randomized metric $\tilde{a}:X\times X\to [0,\infty]$
satsisfies $\tilde{a}(x,y)<\alpha$, one calls the objects $(X,a)$ in ${\bf {\Delta}}\text{-}{\bf Cat}$ {\em probabilistic metric spaces} \cite{HofmannReis, Jager}, and we denote their category by {\bf ProbMet}.

The quantale homomorphisms
$\sigma:[0,\infty]\to {\bf {\Delta}}$
and $\tau:[0,1]\to {\bf {\Delta}}$,
defined by $\sigma(\alpha)(\gamma)=0\text{ if }\gamma\leq\alpha\text{, and }1$ otherwise, and
$\tau(u)(\gamma)=u\text{ if }\gamma>0\text{, and }0$ otherwise, induce full embeddings of {\bf Met} and {\bf ProbOrd} into {\bf ProbMet}, respectively. Their significance lies in the fact that they present $\bf{\Delta}$ as a coproduct of
$[0,\infty]$ and $[0,1]$ in the category of commutative quantales and their homomorphisms, since every $\varphi\in {\bf{\Delta}}$ has a presentation $\varphi={\rm sup}_{\gamma\in[0,\infty]}\sigma(\gamma)\otimes\tau(\varphi(\gamma))$.

\item[\rm{(4)}] The powerset $2^M$ of a (multiplicative) monoid $M$ (with neutral element $e_M$) becomes a quantale when ordered by inclusion and provided with the composition $B\circ A=\{\beta \alpha \mid \alpha, \beta \in M\}$ for $A, B\subseteq M$; in fact, it is the free quantale over the monoid $M$. The objects of $2^M$-{\bf Cat} are sets
$X$ equipped with a family $(\leq_{\alpha})_{\alpha \in M}$ of relations on them satisfying the rules
$x\leq_{e_M}x\text{ and }(x\leq_{\alpha} y,\; y\leq_{\beta}z\Rightarrow x\leq_{\beta \alpha}z)$; morphisms must preserve each relation of the family; see \cite{MonTop} V.1.4. Every homomorphism $\varphi: M\to N$ of monoids may be considered a homomorphism $\varphi:2^M\to 2^N$ of quantales via direct image,
while its right adjoint given by inverse image is in general only a lax homomorphism $\varphi^{-1}:2^N\to 2^M$. Still, 2-functoriality of $ (-){\text{-}}{\bf Cat}$ produces adjunctions $\varphi (-) \dv \varphi^{-1}(-):2^N{\text{-}}{\bf Cat}\to 2^M{\text{-}}{\bf Cat}$. In particular, when considering $1\to M\to 1$ with $1$ trivial, one sees that there is a coreflective embedding of {\bf Ord} into $2^M{\text{-}}{\bf Cat}$, as well as a reflective one.
\item[\rm{(5)}] Every {\em frame, i.e.}, every complete lattice in which binary infima distribute over arbitrary suprema, may be considered a quantale; in fact, these are precisely the commutative quantales in which every element is idempotent. For example, in addition to $\sf 2$ of (1), $([0,\infty],\geq)$ may be considered a quantale $[0,\infty]_{\rm max}$ when, instead of $\alpha +\beta$ as in (2), the binary operation is given by ${\rm max}\{\alpha,\beta\}$.
The resulting category $[0,\infty]_{\rm{max}}{\text{-}}\Cat$ is the category {\bf UMet} of generalized {\em ultrametric} spaces $(X,a)$ whose distance function must satisfy $a(x,z)\leq {\rm max}\{a(x,y), a(y,z)\}$ instead of the weaker triangle inequality.
\end{itemize}
\end{exmp}

A quantale $\sf V$ is called {\em divisible} \cite{Hohle} if for all $u\leq v$ in $\sf V$ there are $a,b\in \sf V$ with $a\otimes v=u=v\otimes b$; it is easy to see that then one may choose $a=u\lda v$ and $ b=v\rda u$. Applying the defining property to $u=\sf k$ and $v=\top$ the top element, so that $\top=\top\otimes {\sf k}=\top\otimes\top\otimes b\leq\top\otimes b=\sf k$, one sees that such a quantale must be {\em integral}, {\em i.e.}, $\sf k=\top$. Of the quantales of Example \ref{firstexamples}, all but $2^M$ are divisible; $2^M$ is not even integral, unless the monoid $M$ is trivial.

We refer to \cite{Freyd, Grandis} for the the Freyd-Grandis construction of freely adjoining a proper orthogonal factorization system to a category. In the case of a quantaloid $\CQ$ it produces the quantaloid ${\sf D}\CQ$ of ``diagonals" of $\CQ$ (so named in \cite{Stubbe2014}, after the prior treatments in \cite{HohleKubiak, PuZhang}), which has a particulary easy description when the quantaloid is a divisible quantale $\sf V$: the objects of  the quantaloid $\sf DV$ are the elements of $\sf V$, and there is a morphism $(u,d,v):u\to v$ in ${\sf DV}$ if $d\in {\sf V}$ satisfies $d\leq u\wedge v$;
for ease of notation, we write $d:u\to/~>/v$, keeping in mind that it is essential to keep track of the domain $u$ and the codomain $v$.
The composite $e\circ d$ of $d$ with $e:v\to/~>/w$ in ${\sf DV}$ is defined by $e\otimes (v\rda d)=(e\lda v)\otimes d$ in $\sf V$, and $v:v\to/~>/v$ serves as the identity morphism on $v$ in {\sf DV}. The order of the hom-sets of {\sf DV} is inherited from {\sf V}.

The quantale {\sf V} is fully embedded into {\sf DV} by the homomorphism
$\iota=\iota_{\sf V}: {\sf V}\to {\sf DV}, \; v\mapsto (v:{\sf k}\to/~>/{\sf k})$,
of quantaloids. There are lax homomorphisms, known as the {\em backward} and {\em forward globalization} functors (see \cite{FourmanScott, PuZhang, TaoLaiZhang}),
    \begin{align*}
    \delta : {\sf DV}\to {\sf V},\quad\quad&\quad\quad\gamma: {\sf DV} \to {\sf V},\\
    (d:u\to/~>/v)\mapsto \;v\rda d\quad&\quad(d:u\to/~>/v) \mapsto d\lda u
    \end{align*}
which, from a factorization perspective, play the role of the {\em domain} and {\em codomain} functors. They satisfy $ \delta\iota_{\sf V}=1_{\sf V}=\gamma\iota_{\sf V}$ and therefore make $\sf V$ a retract of ${\sf DV}$. Consequently, the full embedding
${\sf V}\text{-}{\bf Cat}\to {\sf DV}\text{-}{\bf Cat}$
induced by $\iota$ has retractions, facilitated by $\delta$ and $\gamma$ (see Example \ref{lastexample}).

More importantly, when one considers {\sf V} as a {\sf V}-category ({\sf V,h}) with ${\sf h}(u,v)=v\lda u$, there is a full reflective embedding
$$E_{\sf V}:{\sf DV}\text{-}{\bf Cat}\to{\sf V}\text{-}{\bf Cat}/{\sf V}$$
which provides a {\sf DV}-category $(X,a)$ with the {\sf V}-category structure
 $d$ defined by $d(x,y)=a(x,y)\lda a(x,x)$ and considers it as a
 {\sf V}-category over {\sf V} via $tx=a(x,x)$. Conversely, the reflector provides a
${\sf V}$-category $(X,d)$ that comes equipped with a ${\sf V}$-functor $t:X\to {\sf V}$, with the $\sf DV$-category structure $a$, defined by $a(x,y)=d(x,y)\otimes tx$; see \cite{LaiShenTholen}.

The quantaloids ${\sf DV}$ induced by the divisible quantales $\sf V$ of Example \ref{firstexamples} are of interest in what follows. Here we mention only a couple of easy cases.
\begin{exmp}\label{secondexamples}
\begin{itemize}
\item[\rm(1)]
The quantaloid ${\sf D}{\sf 2}$ has objects $\perp,\top$, and there are exactly two morphisms $\perp,\top:\top\to/~>/\top$ while all other hom-sets are trivial, each of them containing only $\perp$. The object part of a ${\sf D2}$-category structure on a set $X$ is given by its fibre over $\top$, {\em i.e.}, by a subset $A\subseteq X$ and an order on $A$; in other words, by a truly {\em partial} (!) order on $X$. A ${\sf D2}$-functor  $f:(X,A)\to (Y,B)$ is a map $f:X\to Y$ with $f^{-1}B=A$ whose restriction to $A$ is monotone. We write {\bf ParOrd} for {\sf D2}-{\bf Cat}.
\item[\rm(2)]
For a
${\sf D}([0,\infty])$-category $(X,a)$ one must have (in the natural order $\leq$ of $[0,\infty]$) $|x|\leq a(x,x)\leq |x|$ for all
$x\in X$, so that the object part of the structure $a: X\times X\to [0,\infty]$
is determined by its morphism part. Since $\alpha\circ \beta = (\alpha\lda\nu)+\beta=\alpha-\nu+\beta\text{ for }\nu\leq \alpha, \beta\in[0,\infty]$, the defining conditions on $a$ may now be stated as
$$a(x,x)\leq a(x,y)\text{ and }  a(x,z)\leq a(x,y)-a(y,y)+a(y,z)\text{ for all }x,y,z\in X.$$
With ${\sf D}([0,\infty])$-functors $f:(X,a)\to(Y,b)$ required to satisfy
$$b(f(x),f(y))\leq a(x,y)\text{ and }b(f(x),f(x))=a(x,x)\text{ for all }x,y\in X$$
one obtains the category {\bf ParMet} of {\em partial metric spaces}, as originally considered in \cite{BKMP}.
(For example, when one thinks of $a(x,y)$ as of the cost of transporting goods from location $x$ to location $y$, which will entail
some fixed overhead costs $a(x,x)$ and $a(y,y)$ at these locations, the term $-a(y,y)$ in the ``partial triangle inequality" justifies itsself since the operator should not pay the overhead twice at the intermediate location $y$.)
For ${\sf V}=[0,\infty]$, the full embedding $E_{\sf V}$ in fact gives an isomorphism
$${\bf ParMet}\cong {\bf Met}/[0,\infty]$$
of categories; {\em i.e.}, partial metric spaces and their non-expanding maps may equivalently be considered as  metric spaces $(X,d)$ that come with a ``norm" $t:X\to [0,\infty]$ satisfying $ty-tx\leq d(x,y)$ for all $x,y\in X$, the morphisms of which are norm-preserving and non-expanding.
The presentation of {\bf ParMet} as a comma category makes it easy to relate it properly to $\bf Met$, as we may look at the forgetful functor $\Sigma:{\bf Met}/[0,\infty]\to{\bf Met}$ and its right adjoint
$X\mapsto (\pi_2: X\times [0,\infty]\to[0,\infty])$ (with the direct product taken in {\bf Met}). When expressed in terms of partial metrics, $\Sigma$ is equivalently described by
$$B_{\gamma}:{\bf ParMet}\to{\bf Met},\;(X,a)\mapsto(X,\tilde{a}),\;\tilde{a}(x,y)=a(x,y)-a(x,x),$$
and its right adjoint assigns to $(X,d)\in{\bf Met}$ the set $X\times [0,\infty]$ provided with the partial metric $d^{+}$, defined by
$$d^{+}((x,\alpha),(y,\beta))=d(x,y)+{\rm max}\{\alpha,\beta\}$$
for all $x,y\in X, \alpha,\beta\in[0,\infty]$.
\end{itemize}
\end{exmp}

\section{Encoding a quantaloid by its discrete presheaf monad}
For a quantaloid $\CQ$ one forms the category $\QRel$ of $\CQ$-{\em relations}, as follows: its objects are those of $\Set/\CQ_0$, {\em i.e.}, sets $X$ that come with an {\em array} (or {\em type}) map $a=a_X:X\to \CQ_0$, also denoted by $|\text{-}|=|\text{-}|_X$, and a morphism
$\phi:X\rto Y$ in $\QRel$ is given by a family of morphisms $\phi(x,y):|x|_X\to|y|_Y\;(x\in X, y\in Y)$ in $\CQ$; its composite
with $\psi:Y\rto Z$ is defined by
$$(\psi\circ\phi)(x,z)=\bv_{y\in Y}\psi(y,z)\circ\phi(x,y).$$
A map $f:X\to Y$ over $\CQ_0$ may be seen as  a $\CQ$-relation via its $\CQ${\em -graph} or its $\CQ${\em -cograph}, as facilitated by the functors
\begin{align*}
&\Set/\CQ_0\to^{(-)_{\circ}}\QRel\to/<-/^{(-)^{\circ}}(\Set/\CQ_0)^{\op}\\
&f_{\circ}(x,y)=\left\{\begin{array}{ll}
1_{|x|} & \text{if}\ f(x)=y\\
\bot & \text{else}
\end{array}\right\}=f^{\circ}(y,x).
\end{align*}
For $X$ in $\Set/\CQ_0$ and $s\in \CQ_0$, a $\CQ$-{\em presheaf} $\sigma$ on $X$ with array $|\sigma|=s$ is a $\CQ$-relation $\sigma: X \rto \{s\}$ (where $\{s\}$ is considered as a set over $\CQ_0$ via the inclusion map); hence, $\si$ is a family $(\si_x:|x|\to s)_{x\in X}$ of $\CQ$-morphisms
with specified common codomain. Assigning to $X$ the set $\sP X=\sP_{\CQ}X$ of $\CQ$-presheaves on $X$ defines the object part of a left adjoint to (the opposite of) the $\CQ$-cograph functor, with the morphism part and the correspondence under the adjunction described by
$$\bfig
\morphism(-200,100)<400,0>[X`Y;\phi] \place(-10,100)[\mapstochar]
\morphism(-200,-100)<400,0>[Y`\sP X;\phila]
\morphism(-300,20)/-/<650,0>[`;]
\place(760,20)[(\phila(y))_x=\phi(x,y)]
\morphism(1600,100)/@{->}@<4pt>/<700,0>[\Set/\CQ_0`\QRel^{\op};(-)^{\circ}]
\morphism(2300,100)|b|/@{->}@<4pt>/<-700,0>[\QRel^{\op}`\Set/\CQ_0;\sP]
\place(1930,100)[\mbox{\scriptsize$\bot$}]
\morphism(1400,-100)<400,0>[(\sP Y`\sP X);\phi^{\odot}]
\morphism(2040,-100)/|->/<-150,0>[`;]
\morphism(2100,-100)<400,0>[(X`Y).;\phi] \place(2290,-100)[\mapstochar]
\place(1600,-200)[\tau\mapsto\tau\circ\phi]
\efig$$
The unit $\sy$ and counit $\varepsilon$ of the adjunction are given by
$$\sy_X =\overleftarrow{1_X^{\circ}}:X\to \sP X,  \;(\sy y)_x=1_{|y|}\text{ iff }x=y;\quad
\varepsilon_X:X\rto\sP X, \;\varepsilon_X(x,\sigma)=\sigma_x:|x|\to|\sigma|.$$
The adjunction induces the monad $\bbP=\bbP_{\CQ} =(\sP,\sfs,\sy)$ on $\Set/\CQ_0$ which, for future reference, we record here explicitly as well:
$$\bfig
\place(350,-150)[\sP:\Set/\CQ_0 \to \Set/\CQ_0,\quad(X\to^f Y)\mapsto f_!:=(f^{\circ})^{\odot}:\sP X\to\sP Y,\quad (f_!\si)_y=\displaystyle\bv\limits_{x\in f^{-1}y}\si_x;]
\place(1080,-270)[\si\mapsto\si\circ f^{\circ}]
\place(350,-500)[\sfs_X=\varepsilon_X^{\odot}:\sP\sP X\to\sP X,\quad (\sfs_X\Sigma)_x=\displaystyle\bv\limits_{\sigma\in\sP X}\Sigma_{\sigma}\circ\sigma_x.]
\place(100,-620)[ \Sigma\mapsto\Sigma\circ\varepsilon_X]
\efig$$

One notes that  $\CQ\text{-}{\bf Rel}$ is a (large) quantaloid that inherits the pointwise order of its hom-sets from $\CQ$. The full embedding
$\CQ\to\CQ\text{-}{\bf Rel}$, which interprets every $s\in\CQ_0$ as the set $\{s\}$ over $\CQ_0$,
is therefore a homomorphism of quantaloids. Its image serves as a generating set in $\CQ\text{-}{\bf Rel}$.
Under the category equivalence  $\Set/\CQ_0\simeq \Set^{\CQ_0}$ the set $\sP X$ over $\CQ_0$ corresponds to $(\CQ\text{-}{\bf Rel}(X,\{s\}))_{s\in \CQ_0}$, which lives in ${\bf Sup}^{\CQ_0}$. The corresponding order on $\sP X$ is described by
$$\sigma \leq \sigma' \iff |\sigma|=|\sigma'|\text{  and  }\forall x\in X \;(\sigma_x\leq\sigma'_x).$$
For $f:X\to Y$ in $\Set/\CQ_0$, the map $\;f_!:\sP X\to\sP Y$, considered as a morphism in ${\bf Sup}^{\CQ_0}$, preserves suprema and, therefore, has a right adjoint $f^!:\sP Y\to\sP X$ which actually preserves suprema as well and is easily described in $\Set/\CQ_0$ by
$$\forall\tau \in \sP Y,\; x\in X\;((f^!\tau)_x=\tau_{fx});$$
since $$f^!=(f_{\circ})^{\odot},$$ the adjunction $f_!\dashv f^!$ follows from $f_{\circ}\dashv f^{\circ}$ in $\CQ{\text-}{\bf Rel}$ and the monotonicity of $\text{(-)}^{\odot}$ on hom-sets, which we explain next.

The sets ${\bf Set}(Y,\sP X)$ with their pointwise order inherited from $\sP X$ make the bijections
$$\CQ\text{-}{\bf Rel}(X,Y)\to\Set/\CQ_0(Y,\sP X),\;\phi\mapsto\phila,$$
order isomorphisms. Since $$\overleftarrow{\psi\circ\phi}=\phi^{\odot}\cdot\overleftarrow{\psi}$$ (for $\psi:Y\rto Z$), monotonicity of $(\psi\mapsto\psi\circ\phi)$ in $\psi$ makes the maps
$$\Set/\CQ_0(Z,\sP Y)\to\Set/\CQ_0(Z,\sP X), \;g\mapsto \phi^{\odot}\cdot g,$$
monotone. This proves item (1) of the following Lemma.

\begin{lem}\label{first rules}
For $\phi, \phi':X\rto Y$ in $\CQ\text{-}\bf Rel$ and $f:X\to Y,\; g,g':Z\to \sP Y, h:W\to Z$ in $\Set/\CQ_0$ one has:
\begin{itemize}
\item[\rm{(1)}] $\phi\le\phi',\; g\leq g'\Rightarrow \phi^{\odot}\cdot g\cdot h\leq \phi'^{\odot}\cdot g'\cdot h;$
\item[\rm{(2)}] $\sy_X\leq f^!\cdot\sy_Y\cdot f,\; \;f^!\cdot\sfs_Y=\sfs_X\cdot (f^!)_!.$
\end{itemize}
\end{lem}

\begin{proof}
The inequality of (2) follows from the naturality of $\sy$ and the adjunction $f_!\dashv f^!$. For the stated equality, using $$\phi^{\odot}=\sfs_X\cdot(\phila)_!$$
 we can show more generally
 $$\phi^{\odot}\cdot\sfs_Y=\sfs_X\cdot(\phila)_!\cdot\sfs_Y=\sfs_X\cdot\sfs_{\sP X}\cdot (\phila)_{!!}
 =\sfs_X\cdot(\sfs_X)_!\cdot(\phila)_{!!}=\sfs_X\cdot(\phi^{\odot})_!.$$
\end{proof}
Let us finally mention that, of course, there is a functorial dependency of $\sP_{\CQ}$ on the quantaloid $\CQ$, which we may describe briefly, as follows. Let $\vartheta:\CQ\to\CR$ be a lax homomorphism of quantales, and let
$B_{\vartheta_0}:\Set/\CQ_0\to\Set/\CR_0$ be the induced ``discrete change-of-base functor" (as in Section 2). We can then regard $\vartheta$ as a lax natural transformation
$$\vartheta: B_{\vartheta_0}\sP_{\CQ}\to\sP_{\CR} B_{\vartheta_0},$$
so that $$(B_{\vartheta_0}f)_!\cdot\vartheta_X\leq\vartheta_Y\cdot B_{\vartheta_0}(f_!)$$ for all $f:X\to Y$ in $\Set/\CQ_0$; indeed,
for $X\in\Set/\CQ_0$, one defines $\vartheta_X:B_{\vartheta_0}\sP_{\CQ}X\to\sP_{\CR}B_{\vartheta_0}X$ by
$$\sigma=(\sigma_x)_{x\in X}\mapsto\vartheta\sigma=(\vartheta(\sigma_x))_{x\in X}.$$
In fact, $\vartheta$ is now
a lax monad morphism, as described by the following two  diagrams:

\quad

\quad\quad\quad\quad$\bfig
\Atriangle<300,300>[B_{\vartheta_0}`B_{\vartheta_0}\sP_{\CQ}`\sP_{\CR}B_{\vartheta_0};B_{\vartheta_0}\sy^{\CQ}`\sy^{\CR}B_{\vartheta_0}`\vartheta]
\place(300,120)[\geq]
\efig $\quad\quad\quad\quad\quad
 $\bfig
\square/`->`->`->/<1400,300>[B_{\vartheta_0}\sP_{\CQ}\sP_{\CQ}`\sP_{\CR}\sP_{\CR}B_{\vartheta_0}`B_{\vartheta_0}\sP_{\CQ}`\sP_{\CR}B_{\vartheta_0};`B_{\vartheta_0}\sfs^{\CQ}`\sfs^{\CR}B_{\vartheta_0}`\vartheta]
\morphism(0,300)<700,0>[B_{\vartheta_0}\sP_{\CQ}\sP_{\CQ}`\sP_{\CR}B_{\vartheta_0}\sP_{\CQ};\vartheta\sP_{\CQ}]
\morphism(700,300)<700,0>[\sP_{\CR} B_{\vartheta_0}\sP_{\CQ}`\sP_{\CR}\sP_{\CR}B_{\vartheta_0};\sP_{\CR}\vartheta]
\place(700,130)[\geq]
\efig $

\quad

\quad

Note that, if $\CQ,\CR$ are quantales, these properties simplify considerably, since then $B_{\vartheta_0}$ may be treated as being the identity functor of $\Set$. Furthermore, if $\vartheta:\CQ\to\CR$ is a homomorphism of quantaloids,
the lax natural transformation $\vartheta$ becomes strict and makes the two diagrams commute strictly. Consequently, in the quantale case one obtains a morphism ${\mathbb P}_{\CQ}\to{\mathbb P}_{\CR}$ of monads.

We will return to $\vartheta$  as a lax monad morphism in Section 7 where we discuss change-of-base functors in greater generality.

\section{Monads laxly distributing over the presheaf monad, and their lax algebras}

Let $\bbT=(T,m,e)$ be a monad on $\Set/\CQ_0$. We wish to generate certain lax extensions of $\bbT$ to $\CQ\text{-}{\bf Rel}$, {\em i.e.}, to the (dual of the) Kleisli category of the presheaf monad $\bbP_{\CQ}$. Since, as is well known, strict extensions are provided by distributive laws $T\sP\to \sP T$ (see \cite{MonTop}), we should consider a {\em lax distributive law} $\lambda:T\sP\to\sP T$ instead, that is: a family $\lambda_X:T\sP X\to \sP TX\;(X\in \Set/\CQ_0)$ of morphisms in $\Set/\CQ_0$ satisfying the following inequalities for all $f:X\to Y$:

\begin{tabular}{p{0.3cm}p{5.2cm}p{4.5cm}p{4cm}}
(a) & \hskip 0.7cm $\bfig\square(800,0)<600,300>[T\sP X`T\sP Y`\sP TX`\sP TY;T(f_!)`\lam_X `\lam_Y`(Tf)_!]
 \place(1100,150)[\leq]
 \efig$& $(Tf)_!\cdot \lam_X\leq\lam_Y\cdot T(f_!)$ &(lax naturality of $\lambda$);
\end{tabular}

\begin{tabular}{p{0.3cm}p{5.2cm}p{4.5cm}p{4cm}}
(b) & \hskip 0.8cm $\bfig
\Atriangle<300,300>[TX`T\sP X`\sP TX;T\sy_X`\sy_{TX}`\lam_X]
\place(300,120)[\geq]
\efig$ & $\sy_{TX}\leq \lam_X\cdot T\sy_X$  &
(lax ${\mathbb P}_{\CQ}$-unit law);
\end{tabular}

\begin{tabular}{p{0.3cm}p{5.2cm}p{4.5cm}p{4cm}}
(c) & $\bfig
\square/`->`->`->/<1000,300>[T\sP\sP X`\sP\sP T X`T\sP X`\sP T X;`T\sfs_X`\sfs_{T X}`\lam_X]
\morphism(0,300)<500,0>[T\sP\sP X`\sP T\sP X;\lam_{\sP X}]
\morphism(500,300)<500,0>[\sP T\sP X`\sP\sP T X;(\lam_X)_!]
\place(500,130)[\geq]
\efig$ &  $\sfs_{TX}\cdot (\lam_X)_!\cdot \lam_{\sP X}\leq\lam_X\cdot T\sfs_X$  &(lax ${\mathbb P}_{\CQ}$-mult. law);
\end{tabular}

\begin{tabular}{p{0.3cm}p{5.2cm}p{4.5cm}p{4cm}}
(d) & \hskip 0.8cm $\bfig
\Atriangle<300,300>[\sP X`T\sP X`\sP TX;e_{\sP X}`(e_X)_!`\lam_X]
\place(300,120)[\geq] \efig$
& $(e_X)_!\leq\lam_X\cdot e_{\sP X}$ &
(lax ${\mathbb T}$-unit law);
\end{tabular}

\begin{tabular}{p{0.3cm}p{5.2cm}p{4.5cm}p{4cm}}
(e) & $\bfig
\square/`->`->`->/<1000,300>[TT\sP X`\sP TT X`T\sP X`\sP T X;`m_{\sP X}`(m_X)_!`\lam_X]
\morphism(0,300)<500,0>[TT\sP X`T\sP T X;T\lam_X]
\morphism(500,300)<500,0>[T\sP T X`\sP TTX;\lam_{T X}]
\place(500,130)[\geq]
\efig$ & $ (m_X)_!\cdot\lam_{TX}\cdot T\lam_X\leq\lam_X\cdot m_{\sP X}$  &(lax ${\mathbb T}$-mult. law).
\end{tabular}

Each of these laws is said to hold {\em strictly} (at $f$ or $X$) if the respective inequality sign may be replaced by an equality sign; for a {\em strict distributive law}, all lax laws must hold strictly everywhere.

The lax distributive law $\lam$ is called {\em monotone} if
$$f\leq g\Rightarrow\lam_X\cdot Tf\leq\lam_X\cdot Tg$$
for all $f, g :Y\to \sP X$ in $\Set/\CQ_0$. For simplicity, in what follows, we refer to a monotone lax distributive law $\lam:T\sP\to T\sP$ just as a {\em monotone distributive law}, which indirectly emphasizes the fact that the ambient 2-cell structure is given by order; we also say that $\mathbb T$ {\em distributes monotonely} over ${\mathbb P}_{\CQ}$ by $\lam$ in this case, adding {\em strictly} when $\lam$ is strict.

\begin{exmp}\label{laxdistlawexamples}
\begin{itemize}
\item[(1)] For every quantaloid $\CQ$, the identity monad on $\Set/\CQ_0$ distributes strictly and monotonely over $\sP_{\CQ}$, via the identity transformation $1_{\sP}$.
\item[(2)] For every quantale $\sV$, the {\em list}-monad ${\mathbb L}$ on {\bf Set}, {\em i.e.}, the free-monoid monad with underlying {\bf Set}-functor
$LX=\bigcup_{n\geq 0}X^n$, distributes strictly and monotonely over $\sP_{\sV}$, via
$$\otimes_X:L\sP_{\sV} X\to\sP_{\sV}LX,\; (\sigma^1,...,\sigma^n)\mapsto \sigma\text{ with }
\sigma_{(x_1,...,x_m)}=\left\{\begin{array}{ll}
\sigma^1_{x_1}\otimes...\otimes\sigma^n_{x_n}
 & \text{if}\ m=n,\\
\bot & \text{else.}
\end{array}\right\}$$
For $\sV={\sf 2}$, so that $\sP_{\sf 2}\cong \sP$ is the (covariant) power set functor, we in particular obtain the strict monotone distributive law
$$\times_X:LPX\to PLX,\; (A_1,...,A_n)\mapsto A_1\times...\times A_n,$$
that was mentioned in the Introduction.
\item[(3)] For every quantale $\sV$, the {\bf Set}-monad $\mathbb L$ may be extended to ${\bf Set}/\sV$: using the monoid structure of $\sV$, one maps every
$(X,a)\in {\bf Set}/\sV$ to $(LX,\zeta\cdot La)$, with $\zeta:L\sV\to\sV$ the monoid homomorphism with $\zeta(v)=v$, {\em i.e.}, $\zeta:(v_1,...,v_n)\mapsto v_1\otimes...\otimes v_n$.
For the quantaloid $\CQ={\sf DV}$ (as described in Section 2 when $\sV$ is divisible) and
$\mathbb L$ considered as a ${\bf Set}/\sV$-monad, one now obtains a strict monotone distributive law $\otimes :L\sP_{\CQ}\to\sP_{\CQ}L$ defined just as in (2), with the understanding that
$\sigma=\otimes_X(\sigma^1,...,\sigma^n)$ is now given by $\CQ$-arrows
$$\sigma_{(x_1,...,x_m)}:|x_1|\otimes...\otimes|x_m|\to/~>/|\sigma|=|\sigma^1|\otimes...\otimes|\sigma^n|.$$
\item[{(4)}] (See \cite{LaiTholen}.) For every quantale $\sV=(\sV, \otimes, \sf{k})$, the power set monad ${\mathbb P}={\mathbb P}_{\sf 2}$ of {\bf Set} distributes monotonely over ${\mathbb P}_{\sV}$ by the law $\delta:\sP\sP_{\sV}\to\sP_{\sV}\sP$ which, when we write $\sP_{\sV}X=\sV^X$ as the set of maps $X\to\sV$, is defined by
$$\delta_X:\sP(\sV^X)\to\sV^{\sP X},\quad (\delta_X\mathcal{F})(A)=\bigwedge_{x\in A}\;\bigvee_{\sigma \in \mathcal{F}}\sigma(x),$$
for all $\mathcal{F}\subseteq\sV^X, A\subseteq X.$

\item[(5)] Let ${\mathbb U}=(U,\Sigma,\dot{\text{(-)}})$ denote the ultrafilter monad on $\Set$; so, $U$ assigns to a set $X$ the set of ultrafilters on $X$, the unit assigns to a point in $X$ its principal ultrafilter on $X$, and the monad multiplication is given by the so-called Kowalsky sum; see \cite{Manes, Barr, MonTop}.
 For every commutative and {\em completely distributive} quantale $\sV$ (see \cite{Wood, MonTop}), one defines a monotone distributive law $\beta:U\sP_{\sV}\to\sP_{\sV}U$ by
$$\beta_X:U(\sV^X)\to\sV^{UX},\quad (\beta_X\Fz)(\Fx)=\bigwedge_{A\in\Fx,C\in\Fz}\;\bigvee_{x\in A,\sigma\in C}\sigma(x),$$
for all ultrafilters $\Fz$ on $\sV^X$, $\Fx$ on $X$; compare with Corollary IV.2.4.5 of \cite{MonTop}.
\end{itemize}
\end{exmp}

Returning to the general context of a quantaloid $\CQ$ and a monad $\mathbb T$ on ${\bf Set}/\CQ_0$, we define:
\begin{defn}\label{laxalg}
For a monotone distributive law $\lambda: T\sP\to \sP T$, a {\em lax $\lam$-algebra} $(X,p)$ over $\CQ$ is a set $X$ over $\CQ_0$ with a map $p:TX\to \sP X$ over $\CQ_0$ satisfying

\begin{tabular}{p{0.3cm}p{5.4cm}p{3.5cm}p{4cm}}
(f) & \hskip 1.2cm $\bfig
\Atriangle<300,300>[X`TX`\sP X;e_X`\sy_X`p]
\place(300,120)[\geq] \efig$
& $\sy_X\leq p\cdot e_X$ &
(lax unit law); 
\end{tabular}

\begin{tabular}{p{0.3cm}p{5.4cm}p{3.5cm}p{4cm}}
(g) & $\bfig
\square/`->`->`->/<1200,300>[TTX`\sP\sP X`TX`\sP X;`m_X`\sfs_X`p]
\morphism(0,300)<400,0>[TTX`T\sP X;Tp]
\morphism(400,300)<400,0>[T\sP X`\sP TX;\lam_X]
\morphism(800,300)<400,0>[\sP TX`\sP\sP X;p_!]
\place(600,130)[\geq]
\efig$ & $ \sfs_X\cdot p_!\cdot\lam_X\cdot Tp\leq p\cdot m_X$ & (lax mult. law).
\end{tabular}

A {\em lax $\lam$-homomorphism} $f:(X,p)\to(Y,q)$ of lax $\lam$-algebras must satisfy

\begin{tabular}{p{0.3cm}p{5.4cm}p{3.5cm}p{4cm}}
(h) & \hskip 0.8cm $\bfig
\square<800,300>[TX`TY`\sP X`\sP Y;Tf`p`q`f_!]
\place(400,130)[\leq]
\efig$ & $f_!\cdot p\leq q\cdot Tf$ & (lax homomorphism law).
\end{tabular}

The resulting category is denoted by
$ (\lambda,\CQ)\text{-}{\bf Alg}.$
\end{defn}

\begin{exmp}\label{lamalgebraexamples}
\begin{itemize}
\item[(1)] For $\mathbb T$ the identity monad on $\Set/\CQ_0$ and $\lam=1_{\sP_{\CQ}}$,   there is an isomorphism $(\lam,\CQ)\text{-}{\bf Alg}\cong\CQ\text{-}{\bf Cat}$ that commutes with the forgetful functors to $\Set/\CQ_0$.
Indeed, a lax homomorphism $a:X_c\to \CQ$ of quantaloids constitutes a $\CQ$-relation $a:X\rto X$, such that $p=\overleftarrow{a}:X\to\sP X$ satisfies the lax unit- and multiplication laws (f) and (g), and conversely; similarly for the morphisms of the two categories.
\item[(2)] In Section 6 we will elaborate on the correspondence between monotone distributive laws $\lam$ of $\mathbb T$ over ${\sP}_{\CQ}$ and lax extensions $\hat{\mathbb T}$ of the monad
$\mathbb T$ to $\CQ\text{-}{\bf Rel}$. The $\lam$-algebra axioms for a structure $p:TX\to\sP_{\CQ}X$ may then be expressed in terms of a $\CQ$-relation
$X\rto TX$. In the case of $\CQ$ being a commutative quantale $\sV$,
$$(\lam,\sV)\text{-}{\bf Alg}\cong ({\mathbb T},\sV)\text{-}{\bf Cat}$$
becomes the familiar category of $({\mathbb T},\sV)$-{\em categories} $(X,a: TX\rto X)$ (as defined in \cite{MonTop}, but see Remark \ref{comparisonMonTop} below), satisfying the lax-algebra conditions conditions
$${\sf k}\leq a(e_X(x),x)\text{     and     }a({\mathfrak y},z)\otimes \hat{T}a({\mathfrak X},{\mathfrak y})\leq a(m_X({\mathfrak X}),z)$$
for all $z\in X, {\mathfrak y}\in TX, {\mathfrak X}\in TTX$; morphisms, {\em i.e.}, $({\mathbb T},\sV)$-{\em functors}
$f:(X,a)\to (Y,b)$, satisfy $a(x,x')\leq b(fx,fx')$ for all $x,x'\in X$. For example, in the case of Example \ref{laxdistlawexamples}(2), with ${\mathbb T}={\mathbb L}$ and $\sV = {\sf 2}$, one obtains the category
${\bf MulOrd}$ of {\em multiordered sets} $X$ (carrying a reflexive and transitive relation $LX\rto X$).
For $\sV=[0,\infty]$ one obtains the category {\bf MulMet} of {\em multimetric spaces} $(X,a:LX\times X\to[0,\infty])$, defined to satisfy the conditions
\begin{align*}
&a((x),x)=0,\\
&a((\underbrace{x_{1,1},\dots,x_{1,n_1}}_{\Fx_1},\dots,\underbrace{x_{m,1},\dots,x_{m,n_m}}_{\Fx_m}),z)\leq a(\Fx_1,y_1)+\dots+a(\Fx_m,y_m)+a((y_1,\dots,y_m),z);
\end{align*}
morphisms $f:(X,a)\to(Y,b)$ are non-expanding: $b((fx_1,...,fx_n),fy)\leq a((x_1,...,x_n),y)$.
\item[(3)] (See \cite{LaiTholen}.) For any quantale $\sV$ and the monotone distributive law $\delta$ of Example \ref{laxdistlawexamples}(4) that makes the powerset monad ${\mathbb P}={\mathbb P}_{\sf 2}$ distribute over ${\mathbb P}_{\sV}$,
$$(\delta,\sV)\text{-}{\bf Alg}=\sV\text{-}{\bf Cls}$$
is the category of $\sV$-{\em closure spaces} $(X,c:\sP X\to\sV^X)$ which in particular (when $\sV$ is integral), at every ``level" $u\in\sV$, gives rise to the ``$c$-closure"
$A^{(u)}=\{x\in X \;|\;c(A)(x)\geq u\}$ of $A\subseteq X$.
Considering now the full reflective subcategory of $\sV\text{-}{\bf Cls}$ of those spaces $(X,c)$ for which $c$ is a homomorphism of join-semilattices, so that the finite additivity conditions
$$c(\emptyset)=\bot\text{   and   }c(A\cup B)=c(A)\vee c(B)$$
for all $A,B\subseteq X$ are satisfied, one obtains for $\sV={\sf 2},\;[0,\infty],\text{ or }{\bf{\Delta}}$, respectively topological spaces (as described by a closure operation), approach spaces (as described by a point-set distance function \cite{Lowen}), or probabilistic approach spaces \cite{BrockKent}; in general, we call them $\sV$-{\em approach spaces}. Since lax $\delta$-homomorphisms provide the ``right" morphisms in each of the three cases, we denote the resulting category by $\sV\text{-}{\bf App}$ and obtain in the special cases the categories
$${\sf 2}\text{-}{\bf App}={\bf Top},\quad[0,\infty]\text{-}{\bf App}={\bf App},\quad{\bf{\Delta}}\text{-}{\bf App}={\bf ProbApp}.$$
\item[(4)] As shown in \cite{LaiTholen}, for a commutative and completely distributive quantale $\sV$ and the monotone distributive law $\beta$ of Example \ref{laxdistlawexamples}(5) that makes $\mathbb U$ distribute over ${\mathbb P}_{\sV}$,
$$(\beta,\sV)\text{-}{\bf Alg}\cong {\sV}\text{-}{\bf App}$$
is the category of $\sV$-approach spaces; see also Example \ref{algfctrexample}. Considering for $\sV$ the quantales ${\sf 2}, [0,\infty],\text{ and }\bf{\Delta}$, in this way one obtains respectively the ultrafilter characterization of the objects of the categories {\bf Top} of topological spaces (\cite{Barr, MonTop}), {\bf App} of approach spaces (\cite{Lowen, CleHof2003, MonTop}), and {\bf ProbApp} of probabilistic approach spaces (\cite{BrockKent, Zhang, HofmannReis, Jager}).
\end{itemize}
\end{exmp}
In generalization of Proposition \ref{topologicity1} one easily proves:
\begin{prop}\label{topologicity2}
$(\lam,\CQ)\text{-}{\bf Alg}$ is topological over $\Set/\CQ_0$ and, hence, totally complete and totally cocomplete.
\end{prop}

\begin{proof}
For any family of $\lam$-algebras $(Y_i,q_i)$ and $\Set/\CQ_0$-maps $f_i:X\to Y_i\;(i\in I)$, the fixed set
 $X$ obtains its initial structure $p$ with respect to the forgetful functor $(\lam,\CQ)\text{-}{\bf Alg}\to\Set/\CQ_0$ as
$$p:=\bigwedge_{i\in I}(f_i)^!\cdot q_i\cdot Tf_i$$
which, in pointwise terms, reads as $(p{\mathfrak{x}})_x=\bigwedge_{i\in I}(q_i(Tf_i(\Fx)))_{f_ix}$, for all $x\in X, \Fx\in TX$.
\end{proof}

\section{Topological theories and maximal lax distributive laws}
In addition to the given small quantaloid $\CQ$, in this section we restrict ourselves to considering monads ${\mathbb T}$ on $\Set/\CQ_0$ that are liftings of $\Set$-monads along the forgetful $\Sigma:\Set/\CQ_0\to\Set$. The following Proposition (which remains valid when $\Set$ is replaced by an arbitrary category) states that these are completely described by Eilenberg-Moore algebra structures on $\CQ_0$, just as we have encountered them in the special case of the list monad in Example \ref{laxdistlawexamples}(3).
\begin{prop}\label{strict lifting}
Let ${\mathbb T}=(T,m,e)$ be a monad on $\Set$. Then there is a bijective correspondence between ${\mathbb T}$-algebra structures $\zeta:T\CQ_0\to\CQ_0$ and monads $\mathbb T'=(T',m',e')$ on $\Set/\CQ_0$ with
$$\Sigma T'=T\Sigma,\; \Sigma e' =e\Sigma,\; \Sigma m'=m\Sigma.$$
\end{prop}

\begin{proof}
For a ``$\Sigma$-lifting" ${\mathbb T}'$ of $\mathbb T$, let $\zeta$ be the array function of the $\Set/\CQ_0$-object
$T'(\CQ_0,1_{\CQ_0})$, whose domain must necessarily be $T\CQ_0$. For any $\Set/\CQ_0$-object $(X,a)$, the unique $\Set/\CQ_0$-morphism
$a:(X,a)\to (\CQ_0,1_{\CQ_0})$ to the terminal object is being mapped by $T'$ to
$$(TX,a_{TX})\to^{Ta} (T\CQ_0,\zeta),  \text{ so that }
a_{TX}=\zeta\cdot Ta\quad(*).$$
The object assignment by $T'$ is therefore uniquely determined by $\zeta$, and so is its morphism assignment, by faithfulness of $\Sigma$. Furthermore, since necessarily
$$e'_{(X,a)}=e_X:(X,a)\to(TX,\zeta\cdot Ta)\text{  and  }m'_{(X,a)}=m_X:(TTX,\zeta\cdot T\zeta\cdot TTa)\to(TX,\zeta\cdot Ta),$$
in $\Set/\CQ_0$, one has $\zeta\cdot Ta\cdot e_X=a_X$ and $\zeta\cdot Ta\cdot m_X=\zeta\cdot T\zeta\cdot TTa$ which, for $X=\CQ_0$ and $a=1_{\CQ_0}$, amount to the $\mathbb T$-algebra laws $\zeta\cdot e_{\CQ_0}=1_{\CQ_0}$ and $\zeta\cdot m_{\CQ_0}=\zeta\cdot T\zeta$.

Conversely, with $T'$ defined by $(*)$, these laws similarly give the lifting $\mathbb T'$ of $\mathbb T$ along $\Sigma$.
\end{proof}

In what follows, we will not distinguish notationally between $\mathbb T'$ and $\mathbb T$. So, we are working with a $\Set$-monad ${\mathbb T}=(T,m,e)$ and a fixed $\mathbb T$-algebra structure $\zeta:T\CQ_0\to\CQ_0$ on $\CQ_0$ that allows us to treat $\mathbb T$ as a monad on $\Set/\CQ_0$. For such $\mathbb T$ and a monotone distributive law $\lam:T\sP\to \sP T$ we consider the $\Set/\CQ_0$-maps
$$\xi:=(T\sP\CQ_0\to^{\lam_{\CQ_0}}\sP T\CQ_0\to^{\zeta_!}\sP\CQ_0)\text{ and }\theta:=(T\sP\sP\CQ_0\to^{\lam_{\sP\CQ_0}}\sP T\sP\CQ_0\to^{\xi_!}\sP\sP\CQ_0).
$$

\begin{prop}\label{concentration1}
$\xi$ and $\theta$ are lax $\mathbb T$-algebra structures on $\sP\CQ_0$ and $\sP\sP\CQ_0$, respectively, making
$\sy_{\CQ_0}:(\CQ_0,\zeta)\to(\sP\CQ_0,\xi)$ and $\nu_{\CQ_0}:(\sP\sP\CQ,\theta)\to(\sP\CQ_0,\xi)$ lax $\mathbb T$-homomorphisms, that is, producing the following laxly commuting diagrams:

$$\bfig
\qtriangle<600,500>[\sP\CQ_0`T \sP\CQ_0`\sP\CQ_0;e_{\sP\CQ_0}`1_{\sP\CQ_0}`\xi]
\square(1100,0)<700,500>[TT \sP\CQ_0`T\sP\CQ_0`T \sP\CQ_0`\sP\CQ_0;T\xi`m_{\sP\CQ_0}`\xi`\xi]
\square(2200,0)<600,500>[T \CQ_0`T \sP\CQ_0`\CQ_0`\sP\CQ_0;T \sy_{\CQ_0}`\zeta`\xi`\sy_{\CQ_0}]
\place(450,300)[\leq]
\place(1450,250)[\geq]
 \place(2500,250)[\leq]
\efig$$

$$\bfig
\qtriangle<600,500>[\sP\sP\CQ_0`T \sP\sP\CQ_0`\sP\sP\CQ_0;e_{\sP\sP\CQ_0}`1_{\sP\sP\CQ_0}`\theta]
\square(1100,0)<700,500>[TT \sP\sP\CQ_0`T\sP\sP\CQ_0`T \sP\sP\CQ_0`\sP\sP\CQ_0;T\theta`m_{\sP\sP\CQ_0}`\theta`\theta]
\square(2200,0)<600,500>[T \sP\sP\CQ_0`T \sP\CQ_0`\sP\sP\CQ_0`\sP\CQ_0;T \sfs_{\CQ_0}`\theta`\xi`\sfs_{\CQ_0}]
\place(450,300)[\leq]
\place(1450,250)[\geq]
 \place(2500,250)[\leq]
\efig$$
Moreover, $\xi$ ($\theta$) is a strict $\mathbb T$-algebra structure on $\sP\CQ_0$ ($\sP\sP\CQ_0$) if $\lam$ satisfies the lax $\mathbb T$-unit and -multiplication laws  {\rm (d)} and {\rm (e)} strictly at $\CQ_0$ (at $\sP\CQ_0$, respectively);
and $\sy_{\CQ_0}$ ($\sfs_{\CQ_0}$) is a strict $\mathbb T$-homomorphism if $\lam$ satisfies the lax ${\mathbb P}_{\CQ}$-unit law {\rm (b)} (the lax ${\mathbb P}_{\CQ}$-multiplication law {\rm (c)}, respectively) strictly at $\CQ_0$.
\end{prop}

\begin{proof}
Lax unit law for $\xi$: By (d),
$\xi\cdot e_{\sP\CQ_0}\geq\zeta_!\cdot (e_{\CQ_0})_!=(\zeta\cdot e_{\CQ_0})_!=1_{\sP\CQ_0}$, with equality holding when $\lam$ satisfies (d) strictly at $\CQ_0$. Lax multiplication law for $\xi$: By (e),
 $$\xi\cdot m_{\sP\CQ_0}
  \geq\zeta_!\cdot (m_{\CQ_0})_!\cdot\lam_{T\CQ_0}\cdot T\lam_{\CQ_0}
  =\zeta_!\cdot (T\zeta)_!\cdot \lam_{T\CQ_0}\cdot T\lam_{\CQ_0}       =\zeta_!\cdot\lam_{\CQ_0}\cdot T(\zeta_!)\cdot T\lam_{\CQ_0}
  =\xi\cdot T\xi,$$
  with equality holding when $\lam$ satisfies (e) strictly at $\CQ_0$.

  One proceeds similarly for the (lax) unit and multiplication laws for $\theta$.

  Lax homomorphism law for $\sy_{\CQ_0}$: By (b),  $\xi\cdot T\sy_{\CQ_0}\geq \zeta_!\cdot\sy_{T\CQ_0}=\sy_{\CQ_0}\cdot\zeta,$  with equality holding when $\lam$ satisfies (b) strictly at $\CQ_0$. Lax homomorphism law for $\sfs_{\CQ_0}$: By (c),
  $$\xi\cdot T\sfs_{\CQ_0}\geq \zeta_!\cdot\sfs_{T\CQ_0}\cdot(\lam_{\CQ_0})_!\cdot\lam_{\sP\CQ_0}
  =\sfs_{\CQ_0}\cdot\zeta_{!!}\cdot(\lam_{\CQ_0})_!\cdot\lam_{\sP\CQ_0}
  =\sfs_{\CQ_0}\cdot\xi_!\cdot\lam_{\sP\CQ_0}=\sfs_{\CQ_0}\cdot T\sy_{\CQ_0},
   $$ with equality holding when $\lam$ satisfies (c) strictly at $\CQ_0$.
\end{proof}

\begin{rem}\label{t}
\begin{itemize}
\item[(1)] Let $t:=|\text{-}|_{\sP\CQ_0}$ denote the array map of $\sP\CQ_0$ (that assigns to a $\CQ_0$-indexed family of $\CQ$-morphisms in $\sP\CQ_0$
their common codomain). Then $|\text{-}|_{T\sP\CQ_0}=\zeta\cdot Tt$ (see ($*$) of Proposition \ref{strict lifting}), and since $\xi$ is a map over $\CQ_0$, we must have $t\cdot \xi=\zeta \cdot Tt$. In other words,
$t: (\sP\CQ_0,\xi)\to(\CQ_0,\zeta)$ is a strict $\mathbb T$-homomorphism.
\item[(2)] From $\xi=\zeta_!\cdot\lam_{\CQ_0}$ one obtains $\lam_{\CQ_0}\leq \zeta^!\cdot\xi$ by adjunction, and the lax naturality (a) of $\lam$ at $t$ then gives
$$\lam_{\sP\CQ_0}\leq(Tt)^!\cdot\lam_{\CQ_0}\cdot T(t_!)\leq (Tt)^!\cdot\zeta^!\cdot\xi\cdot T(t_!).$$
Consequently, one obtains an upper bound for $\theta$:
$$\theta=\xi_!\cdot\lam_{\sP\CQ_0}\leq\xi_!\cdot  (Tt)^!\cdot\zeta^!\cdot\xi\cdot T(t_!).$$
\end{itemize}
\end{rem}

We now embark on a converse path, by establishing a monotone distributive law from a given map $\xi$, in addition to $\zeta$, and by choosing $\theta$ maximally.

\begin{defn}\label{toptheory}
Let $\mathbb T$ be a $\Set$-monad that comes with a $\mathbb T$-algebra structure $\zeta$ on the object set $\CQ_0$
of the small quantaloid $\CQ$. A {\em topological theory} for $\mathbb T$ and $\CQ$ is a $\Set$-map $\xi: T\sP\CQ_0\to\sP\CQ_0$ satisfying the following conditions:

\begin{tabular}{lr}
&\\
0. $t\cdot\xi=\zeta\cdot Tt$ \quad\quad\quad\quad\quad\quad\quad\quad\quad\quad\quad(with $t$ as in Remark \ref{t}(1))\ & (array compatibility);\\
1. $1_{\sP\CQ_0}\leq\xi\cdot e_{\sP\CQ_0}, \;\xi\cdot T\xi\leq\xi\cdot m_{\sP\CQ_0}$ & (lax $\mathbb T$-algebra laws);\\
2. $\sy_{\CQ_0}\cdot\zeta\leq\xi\cdot T\sy_{\CQ_0}, \; \sfs_{\CQ_0}\cdot\theta\leq\xi\cdot T\sfs_{\CQ_0}$ (with $\theta=\xi_!\cdot  (\zeta\cdot Tt)^!\cdot\xi\cdot T(t_!))$& (lax $\mathbb T$-homom. laws);\\
3. $\forall f, g: Y\to\sP\CQ_0\text{ in  }\Set/\CQ_0\;(f \leq g\Rightarrow\xi\cdot Tf\leq\xi\cdot Tg)$ & (monotonicity).\\
&\\
\end{tabular}
The theory is {\em strict} if the inequality signs in conditions 1 and 2 may be replaced by equality signs.
\end{defn}

Proposition \ref{concentration1} produces for every (strict) monotone distributive law a (strict) topological theory. We will call this theory {\em induced} by the given law.

\begin{thm}\label{concentration2}
For ${\mathbb T}, \CQ, \zeta$ as in $\rm{Definition}\text{ } \ref{toptheory}$ and a topological theory $\xi$,
$$\lam_X^{\xi}:=(\zeta\cdot Ta)^!\cdot\xi\cdot T(a_!)$$
for all $X=(X,a)\in \Set/\CQ_0$ defines a monotone distributive law $\lam^{\xi}$ for $\mathbb T$ and $\CQ$. This law is largest amongst all laws that induce the given theory $\xi$.
\end{thm}

\begin{proof}
We check monotonicity of $\lam=\lam^{\xi}$ and each of the conditions (a)-(e), considering $f:(X,a)\to (Y,b)$ in ${\bf Set}/\CQ_0$. Note that $c:= \zeta\cdot Ta$ is the array function of $TX$.
With $t$ the array function of $\sP\CQ_0$ (see Remark \label{t}(1)), an easy inspection shows that $s:=t\cdot a_!$ is the array function of $\sP X$.

Monotonicity: For $g, h: Y\rto \sP X$ in $\Set/\CQ_0$, monotonicity of $\xi$ gives
$$\lam_X\cdot Tg=c^!\cdot\xi\cdot T(a_!\cdot g)\leq c^!\cdot\xi\cdot T(a_!\cdot h)=\lam_X\cdot Th.$$

(a) With the adjunction $(Tf)_!\dashv(Tf)^!$, from $b\cdot f=a$ one obtains $(Tf)_!\cdot (Ta)^!\leq(Tb)^!$. Hence,
$$(Tf)_!\cdot\lam_X=(Tf)^!\cdot(Ta)^!\cdot\zeta^!\cdot\xi\cdot T(a_!)\leq(Tb)^!\cdot\zeta^!\cdot\xi\cdot T(b_!)\cdot T(f_!)=\lam_Y\cdot T(f_!).$$

(b) Condition 2 for a lax topological theory and Lemma \ref{first rules}(2) give
$$\lam_X\cdot T\sy_X=c^!\cdot\xi\cdot T(a_!)\cdot T\sy_X=
c^!\cdot\xi\cdot T\sy_{\CQ_0}\cdot Ta\geq c^!\cdot\sy_{\CQ_0}\cdot c\geq\sy_{TX}.$$

(c) The adjunction $(T(a_!))_!\dashv(T(a_!))^!$ gives $(T(a_!))_!\cdot(Ts)^!\leq(Tt)^!$. Hence,  with condition 2 for a lax topological theory and Lemma \ref{first rules}(2) one obtains
\begin{align*}
\lam_X\cdot T\sfs_X&=c^!\cdot\xi\cdot T(a_!)\cdot T\sfs_X=c^!\cdot\xi\cdot T\sfs_{\CQ_0}\cdot T(a_{!!})\\
&\geq c^!\cdot\sfs_{\CQ_0}\cdot\theta\cdot T(a_{!!})=\sfs_{TX}\cdot (c^!)_!\cdot\theta\cdot T(a_{!!})\\
&=\sfs_{TX}\cdot(c^!)_!\cdot\xi_!\cdot(Tt)^!\cdot\zeta^!\cdot\xi\cdot T(t_!)\cdot T(a_{!!})\\
&\geq\sfs_{TX}\cdot (c^!)_!\cdot\xi_!\cdot(T(a_!))_!\cdot(Ts)^!\cdot\zeta^!\cdot\xi\cdot T(s_!)\\
&=\sfs_{TX}\cdot (c^!)_!\cdot\xi_!\cdot(T(a_!))_!\cdot(\zeta\cdot Ts)^!\cdot\xi\cdot T(s_!)\quad=\sfs_{TX}\cdot(\lam_X)_!\cdot\lam_{\sP X}.
\end{align*}

(d) From $\zeta\cdot e_{\CQ_0}=1_{\CQ_0}$ one obtains $(e_{\CQ_0})_!\leq\zeta^!$ by adjunction. Together with condition 3 for a lax topological theory, this gives
\begin{align*}
\lam_X\cdot e_{\sP X}&=c^!\cdot\xi\cdot T(a_!)\cdot e_{\sP X}=c^!\cdot\xi\cdot e_{\sP\CQ_0}\cdot a_!\\
&\geq(\zeta\cdot Ta)^!\cdot a_!\geq (Ta)^!\cdot(e_{\CQ_0})_!\cdot a_!=(Ta)^!\cdot (Ta)_!\cdot(e_X)_!\quad\geq (e_X)_!.
\end{align*}

(e) With $d:=\zeta\cdot Tc$ the array function of $TTX$, from $c\cdot m_X=d$ one obtains $(m_X)_!\cdot d^!\leq c^!$ by adjunction, so that condition 3 for a topological theory gives

\begin{align*}
\lam_X\cdot m_{\sP X}&=c^!\cdot\xi\cdot T(a_!)\cdot m_{\sP X}=c^!\cdot\xi\cdot m_{\sP\CQ_0}\cdot TT(a_!)\\
&\geq c^!\cdot\xi\cdot T\xi\cdot TT(a_!)\geq (m_X)_!\cdot d^!\cdot\xi\cdot T\xi\cdot TT(a_!)\\
&\geq(m_X)_!\cdot d^!\cdot \xi\cdot T(c_!)\cdot T(c^!)\cdot T(\xi)\cdot TT(a_!)\quad=(m_X)_!\cdot\lam_{TX}\cdot T\lam_X.
\end{align*}

Next we show that the topological theory $\xi'$ induced by $\lam=\lam^{\xi}$ equals $\xi$. Indeed, since $\zeta$
is surjective, one has $\zeta_{\circ}\circ\zeta^{\circ}=1_{T\CQ_0}$ and therefore
$$\xi'=\zeta_!\cdot\lam_{\CQ_0}=\zeta_!\cdot\zeta^!\cdot\xi=(\zeta_{\circ}\circ\zeta^{\circ})^{\odot}\cdot\xi=\xi.$$

Finally, let $\kappa:T\sP\to\sP T$ be any monotone distributive law inducing $\xi$, so that $\zeta_!\cdot\kappa_{\CQ_0}=\xi$. Then
$$\lam_X=c^!\cdot\xi\cdot T(a_!)=(Ta)^!\cdot\zeta^!\cdot\zeta_!\cdot\kappa_{\CQ_0}\cdot T(a_!)\geq(Ta)^!\cdot\kappa_{\CQ_0}\cdot T(a_!)=(Ta)^!\cdot T(a_!)\cdot\kappa_X\geq\kappa_X.$$
\end{proof}

\begin{rem}
(1) When stated in pointwise terms, the definition of $\lam=\lam^{\xi}$ reads as
$$(\lam_X\Fz)_{\Fx}=(\xi\cdot T(a_!)(\Fz))_{\zeta\cdot Ta(\Fx)},$$
for all $X=(X,a)\in\Set/\CQ_0,\Fx \in TX, \Fz\in T\sP X$.

(2) {\em For a topological theory $\xi$, the structure $\theta$ as in} Definition \ref{toptheory} {\em always satisfies the lax $\mathbb T$-unit and -multiplication laws of} Proposition \ref{concentration1}, since $\xi$ is induced by the monotone distributive law $\lam^{\xi}$.
\end{rem}

\begin{cor}\label{Galois}
For a quantaloid $\CQ$ and a $\Set$-monad $\mathbb T$ that comes equipped with a $\mathbb T$-algebra structure $\zeta$ on the set of objects of $\CQ$, the assignments
$$(\xi\mapsto \lam^{\xi})\text{   and   }(\lam\mapsto \xi^{\lam}:=\zeta_!\cdot\lam_{\CQ_0})$$
define an adjunction between the ordered set of topological theories for $\mathbb T$ and $\CQ$ and the conglomerate of monotone distributive laws $T\sP_{\CQ}\to\sP_{\CQ}T$, ordered componentwise.
\end{cor}

\begin{defn}\label{maxlaw}
A monotone distributive law $\lam$ is {\em maximal} if it is closed under the correspondence of Corollary \ref{Galois}, that is, if it is induced by some topological theory or, equivalently, by
$\xi^{\lam}$. More explicitly then, $\lam$ is maximal if, and only if, for all $X=(X,a)\in \Set/\CQ_0$,
$$\lam_X=(Ta)^!\cdot\zeta^!\cdot\zeta_!\cdot\lam_{\CQ_0}\cdot T(a_!).$$
Note that this condition simplifies to $\lam_X=(Ta)^!\cdot\lam_{\CQ_0}\cdot T(a_!)$ when $\zeta$ is bijective.
\end{defn}

\begin{cor}\label{bijectiveGalois}
Maximal monotone distributive laws correspond bijectively to topological theories.
\end{cor}

\begin{exmp}\label{inducedxiexamples}
\begin{itemize}
\item[(1)] For $\mathbb T$ and $\lam$ identical (as in Example \ref{laxdistlawexamples}(1)), with $\zeta=1_{\CQ_0}$ also the induced map $\xi=1_{\sP\CQ_0}$ is identical, but the maximal law $\lam^{\xi}$ associated with it (by Theorem \ref{concentration2}) is not; for a set $X$ with array function $|\text{-}|:X\to\CQ_0$ one has
$$\lam^{\xi}_X:\sP_{\CQ}X\to\sP_{\CQ}X, \;(\lam_X^{\xi}\sigma)_y=\bigvee\{\sigma_x\;|\;x\in X, |x|=|y|\},$$
for all $\sigma\in\sP X, y\in X.$
\item[(2)] For ${\mathbb T}={\mathbb L}$ and the strict distributive law $\otimes$ of Example \ref{laxdistlawexamples}(2), the induced map $\xi:L\sV\to\sV$ with $ (v_1,...,v_n)\mapsto v_1\otimes ...\otimes v_n$ is in fact the Eilenberg-Moore structure of the monoid $(\sV,\otimes,{\sf k})$.
The maximal law $\lam^{\xi}_X:L(\sV^X)\to\sV^{LX}$ maps $(\sigma_1,...,\sigma_n )$ to the map $LX\to\sV$
with constant value $\bigvee\{\sigma_1(z_1)\otimes ...\otimes
\sigma_n(z_n)\;|\;z_1,...,z_n\in X\}$, for every set $X$.
\item[(3)] For $\CQ={\sf D}\sV$ with $\sV$ divisible, $ {\mathbb T}={\mathbb L}$, and the distributive law and the map $\zeta:L\sV\to \sV$
as in Example \ref{laxdistlawexamples}(3) (which coincides with the map $\xi$ of (2) above), the now induced map
$\xi:L(\sP\sV)\to\sP\sV=\sP_{\CQ}\sV$ is given by
$$(\xi(\sigma^1,...,\sigma^n)) _u=\bigvee_{v_1\otimes...\otimes v_n=u}\sigma_{v_1}^1\otimes ...\otimes\sigma_{v_n}^n:\quad u\to/~>/|\sigma^1|\otimes...\otimes |\sigma^n|,$$
for all $\sigma^1,...,\sigma^n\in \sP\sV, u\in\sV$.
\item[(4)] The map $\xi:\sP\sV\to\sV$ induced by the law $\delta$ of Example \ref{laxdistlawexamples}(4) has constant value $\top$.
\item[(5)] The map $\xi: U\sV\to \sV$ induced by the ultrafilter monad and the law $\beta$ as in Example \ref{laxdistlawexamples}(5) is given by
$$\xi(\Fz)=\bigwedge_{C\in \Fz}\bigvee C\quad\quad(=\bigvee_{C\in\Fz}\bigwedge C,\text{   if }\sV\text{ is completely distributive}),$$
for every ulltrafilter $\Fz$ on $\sV$; it plays a central role in \cite{Hofmann}.

\end{itemize}
\end{exmp}

While typically maximal monotone distributive laws are rather special and often allow only for trivial $\lam$-algebras, especially when $\CQ$ is a quantale (see Remark \ref{final remark}(2)), they do lead to interesting categories $(\lam,\CQ)\text{-}{\bf Alg}$
 when $\CQ$ is a multi-object quantaloid, including the case when $\CQ={\sf DV}$ for a quantale $\sV$. We can mention here only
the easiest case.

\begin{exmp}\label{arrayinvariant}
Consider the maximal law $\lam=\lam^{\xi}$ induced by the identity map $\xi=1_{\sP\CQ_0}$ of Example \ref{inducedxiexamples}(1), for any quantaloid $\CQ$ and $\mathbb T$ the identity monad on $\Set/\CQ_0$. Writing
$a(x,y):=(py)_x$ for $x,y\in X$ and a lax $\lam$-algebra structure $p:X\to\sP X$ on a set $X$ with array map $|\text{-}|:X\to\CQ_0$, conditions (f), (g) of Definition \ref{laxalg} translate to
$$1_{|x|}\leq a(x,x)\text{   and   }(|y|=|y'|\Longrightarrow a(y',z)\circ a(x,y)\leq a(x,z))$$
for all $x,y,y',z\in X$. Since in particular $a(y,y)\circ a(x,x)\leq a(x,y)$ whenever $|x|=|y|$, these conditions are equivalent to
$$(|x|=|y|\Longrightarrow 1_{|x|}\leq a(x,y))\text{   and   }a(y,z)\circ a(x,y)\leq a(x,z)$$
for all $x,y,z\in X$. Consequently then, $(\lam,\CQ)\text{-}{\bf Alg}$ can be seen as the full subcategory of $\CQ\text{-}{\Cat}$ containing those $\CQ$-categories $(X,a)$ satisfying $1_{|x|}\leq a(x,y))$ for all $x,y$ with the same array. In the case of
$\CQ={\sf D}[0,\infty]$ (see Example \ref{secondexamples}(2)), this is the full subcategory of {\bf ParMet} of those partial metric spaces $(X,a)$ satisfying the array-invariance condition
$$a(x,x)=a(y,y)\Longrightarrow a(x,y)=a(x,x)$$
for all $x,y\in X$.

\end{exmp}

\section{Lax distributive laws of ${\mathbb T}$ over $\mathbb P_{\CQ}$ versus lax extensions of $\mathbb T$ to $\CQ\text{-}{\bf Rel}$}

In this section we give a precise account of the bijective correspondence between monotone distributive laws of $\mathbb T$ over ${\mathbb P}_{\CQ}$ and so-called lax extensions of $\mathbb T$ to $\CQ\text{-}{\bf Rel}$, {\em i.e.}, to the Kleisli category of ${\mathbb P}_{\CQ}$, where $\mathbb T$ is now again an arbitrary monad of $\Set/\CQ_0$, {\em i.e.}, not necessarily a lifting of a $\Set$-monad as in Section 5.

\begin{rem}\label{rules}
For future reference, we give a list of identities that will be used frequently in what follows. In part they have already been used in Section 3, and they all follow from the discrete presheaf adjunction that induces ${\mathbb P}_{\CQ}$. For morphisms $\varphi:X\rto Y,\;\psi:Y\rto Z$ in $\CQ\text{-}{\bf Rel}$ and $f:X\to Y, \;g:X\to Z,\;h:Z\to Y$ in $\Set/\CQ_0$ one has:
\begin{itemize}
\item[\rm{(1)}] $\phila=\varphi^\odot\cdot\sy_Y,\quad \varphi=\phila^{\circ}\circ\varepsilon_X,\quad\phi^{\odot}=\sfs_X\cdot\phila_!,\quad(\phi^{\odot})^{\circ}=\varepsilon_Y\circ\phi;$
\item[\rm{(2)}] $\overleftarrow{\psi\circ\phi}=\phi^{\odot}\cdot\overleftarrow{\psi},\quad g_!\cdot\phila=\overleftarrow{\varphi\cdot g^{\circ}},\quad\overleftarrow{h^{\circ}\circ\phi}=\phila\cdot h;$
\item[\rm{(3)}] $\overleftarrow{f^{\circ}}=\sy_Y\cdot f=f_!\cdot \sy_X,\quad\overleftarrow{1_X^{\circ}}=\sy_X,\quad 1_X^{\circ}=\sy_X^{\circ}\circ\varepsilon_X, \quad\overleftarrow{\varepsilon_X}=1_{\sP X}.$
\end{itemize}
\end{rem}

 In what follows, we analyze which of the inequalities required for lax extensions and distributive laws correspond to each other, starting with the most general scenario.
Hence, initially we consider mere families $\lambda_X:T\sP X\to \sP TX\;(X\in \Set/\CQ_0)$ of maps in $\Set/\CQ_0$, which we will call $(\bbT,\CQ)$-{\em distribution families}, and contrast them with families $\hT\phi:T X\rto T Y\;(\phi:X\rto Y\text{ in }\CQ\text{-}{\bf Rel})$, which we refer to as $(\bbT,\CQ)$-{\em extension families}. Certainly, a distribution family $\lambda=(\lambda_X)_X$ determines an extension family
$$\Phi(\lambda)=\hT=(\hT \phi)_\phi\text{  with  }
\overleftarrow{\hT\phi}:=\lambda_X\cdot T\overleftarrow{\phi},$$
also visualized by
$$\begin{array}{ccc}
(\phi:X\rto Y) & \mapsto & (\hT\phi:T X\rto T Y)\\
(\phila:Y\to\sP X) & \mapsto & \bfig\Vtriangle/->`->`<-/<300,300>[T Y`\sP T X`T\sP
X.;\overleftarrow{\hT\phi}`T\phila`\lam_X]\efig
\end{array}$$
We see immediately that we may retrieve $(\lambda_X)_X$ from $(\hT \phi)_{\phi}$, by choosing $\phi$ such that $\overleftarrow{\phi}=1_{\sP X}$,
 which is the case precisely when $\phi = \varepsilon_X:X\rto \sP X$ (the co-unit of the adjunction presented in Section 3). Hence, when assigning to any extension family $\hT=(\hT\phi)_{\phi}$ the distribution family

 $$\Psi (\hT)=\lambda=(\lambda_X)_X\text{   with   }\lambda_X:=\overleftarrow{\hT\varepsilon_X},$$
 we certainly have $\Psi\Phi(\lambda)=\lambda$ for all distribution families $\lambda$. The following Proposition clarifies which extension families correspond bijectively to distribution families. We call an extension family $\hT$ {\em monotone} if it satisfies
 $$\forall\phi,\phi':X\rto Y\quad(\phi\leq\phi'{}\Lra{}\hT\phi\leq \hT\phi'),$$
and monotonicity of a lax distribution family is defined as monotonicity for a lax distributive law in Section 4.

 \begin{prop}\label{corr}
 $\Phi$ and $\Psi$ establish a bijective correspondence between all $(\bbT,\CQ)$-distribution families and those $(\bbT,\CQ)$-extension families $\hT=(\hT\phi)_{\phi}$ which satisfy the {\rm left-op-whiskering} condition

 $${\rm(0)}\quad\hT(h^{\circ}\circ\phi)=(T h)^{\circ}\circ\hT\phi$$
 for all $\phi:X\rto Y\text{ in }\CQ\text{-}{\bf Rel}, h:Z\to Y\text{ in }\Set/\CQ_0$. The correspondence restricts to a bijective correspondence between the conglomerate $(\bbT,\CQ)\text{-}\rm{DIS}$ of all monotone distribution families and the conglomerate $(\bbT,\CQ)\text{-}\rm{EXT}$ of all monotone extension families satisfying $\rm{(0)}$.
  \end{prop}

  \begin{proof}
  For a distribution family $\lambda$ and $\hT:=\Phi(\lambda)$, let us first verify the identity (0), using the definition
  of $\hT$ and Remark \ref{rules}(2):

 $$\overleftarrow{(Th)^{\circ}\circ\hT\phi}
  =\overleftarrow{\hT\phi}\cdot Th
  =\lam_X\cdot T\phila\cdot Th
  =\lam_X\cdot T(\overleftarrow{h^{\circ}\circ\phi})
  =\overleftarrow{\hT(h^{\circ}\circ\phi)}.$$
  Monotonicity of $\hT$ follows trivially from the corresponding property of $\lam$.

  Next, for any extension family $\hT$ satisfying (0), we must show $\Phi\Psi (\hT)=\hT$. Indeed,
  with $\lam:=\Psi(\hT)$, the definition of $\Phi(\lam)$ and Remark \ref{rules}(1) give

  \begin{tabular}{ll}
  &\\
  $\overleftarrow{(\Phi\Psi(\hT))\phi}$ & $=\lam_X\cdot T\phila
  =\overleftarrow{\hT\varepsilon_X}\cdot T\phila
  =(\hT\varepsilon_X)^{\odot}\cdot\sy_{T\sP X}\cdot T\phila$\\
  & $=(\hT\varepsilon_X)^{\odot}\cdot((T\phila)^{\circ})^{\odot}\cdot \sy_{TY}
  =(\hT(\phila^{\circ}\circ\varepsilon_X))^{\odot}\cdot\sy_{TY}
  =(\hT\phi)^{\odot}\cdot\sy_{TY}=\overleftarrow{\hT\phi}$.\\
  &\\
  \end{tabular}

That monotonicity of  $\lam$ follows from the monotonicity of $\hT$ and (0) is clear once one has observed that
$$\lam_X\cdot Tf=\overleftarrow{\hT\varepsilon_X}\cdot Tf=\overleftarrow{(Tf)^{\circ}\cdot\hT\varepsilon_X}=
\overleftarrow{\hT(f^{\circ}\cdot\varepsilon_X)}$$
for all $f:Y\to\sP X$  in  $\Set/\CQ_0$.
  \end{proof}

 Before pursuing the bijective correspondence further, let us contrast condition (0) with some other natural conditions for an extension family, as follows.

 \begin{prop}\label{extfamrules}
 Let the monotone extension family $\hT$ satisfy $\hT\psi\circ\hT\phi\leq\hT(\psi\circ\phi)$ for all $\phi,\psi\in\CQ\text{-}\bf{Rel}$. Then the following conditions are equivalent when quantified over the variables occurring in them (with maps $f:X\to Y, \;h:Z\to Y$ over $\CQ_0$):
 \begin{itemize}
 \item[\rm{(i)}] $1_{TX}^{\circ}\leq\hT(1_X^{\circ}), \quad\hT(h^{\circ}\circ\phi)=(Th)^{\circ}\circ\hT\phi$;
 \item[\rm{(ii)}] $1_{TX}^{\circ}\leq\hT(1_X^{\circ}), \quad\hT(\psi\circ f_{\circ})=\hT\psi\circ(Tf)_{\circ}$;
 \item[\rm{(iii)}] $(Tf)^{\circ}\leq\hT(f^{\circ}), \quad(Tf)_{\circ}\leq\hT(f_{\circ})$.
  \end{itemize}
 \end{prop}

 \begin{proof}
$\rm{(i)}\Rightarrow\rm{(iii)}$ The hypotheses, the adjunction $f_{\circ}\dashv f^{\circ}$, and the  monotonicity give
$$1_{TX}^{\circ}\leq\hT(1_X^{\circ})\leq\hT(f^{\circ}\circ f_{\circ})=(Tf)^{\circ}\circ\hT(f_{\circ}),$$
so that $(Tf)_{\circ}\leq \hT(f_{\circ})$ follows with the adjunction $(Tf)_{\circ} \dashv (Tf)^{\circ}$. Furthermore,
$$(Tf)^{\circ}=(Tf)^{\circ}\circ 1_{TY}\leq(Tf)^{\circ}\circ\hT(1_Y^{\circ})=\hT(f^{\circ}\circ 1_Y^{\circ})=\hT(f^{\circ}).$$

$\rm{(iii)}\Rightarrow\rm{(i)}$ One uses (iii) and the general hypotheses on $\hT$ to obtain:
\begin{align*}
(Th)^{\circ}\circ\hT\phi&\leq\hT(h^{\circ})\circ\hT\phi\leq\hT(h^{\circ}\circ\phi)\\
&\leq(Th)^{\circ}\circ(Th)_{\circ}\circ\hT(h^{\circ}\circ\phi)\leq(Th)^{\circ}\circ\hT(h_{\circ})\circ\hT(h^{\circ}\circ\phi)\\
&\leq(Th)^{\circ}\circ\hT(h_{\circ}\circ h^{\circ}\circ\phi)\leq(Th)^{\circ}\circ\hT\phi.
\end{align*}

$\rm{(i)}\Leftrightarrow\rm{(ii)}$: One proceeds analogously to $\rm{(i)}\Leftrightarrow\rm{(iii)}$.
 \end{proof}

 In what follows we compare the conditions on  $\lambda\in (\bbT,\CQ)\text{-}\rm{DIS}$ encountered in Section 4 with some relevant conditions on the related family $\hT\in (\bbT,\CQ)\text{-}\rm{EXT}$ under the
   correspondence of Proposition \rm{\ref{corr}}, so that $\hT = \Phi(\lambda), \lambda = \Psi(\hT),$
   all to be read as quantified over all new variables $(\phi:X\rto Y, \psi:Y\rto Z, f:X\to Y, g:Y\to X)$ occurring in them.

\begin{tabular}{p{0.5cm}p{6cm}p{0.3cm}p{5.2cm}}
(a) & \hskip 0.5cm $\bfig\square(800,0)<700,500>[T\sP X`T\sP Y`\sP TX`\sP TY;T(f_!)`\lam_X `\lam_Y`(Tf)_!]
\place(1150,280)[\leq]
\efig$ & (1) & $\hT\psi\circ (Tg)^{\circ}\leq \hT(\psi\circ g^{\circ})$
\end{tabular}

\

\begin{tabular}{p{0.5cm}p{6cm}p{0.3cm}p{5.2cm}}
(b) & \hskip 0.8cm $\bfig
\Atriangle<300,300>[TX`T\sP X`\sP TX;T\sy_X`\sy_{TX}`\lam_X]
\place(300,120)[\geq]
\efig$ & (2) & $1_{TX}^{\circ}\leq\hT(1_X^{\circ})$\\
& & (2') & $(Tf)^{\circ}\leq \hT(f^{\circ})$
\end{tabular}

\begin{tabular}{p{0.5cm}p{6cm}p{0.3cm}p{5.2cm}}
(c) & $\bfig
\square/`->`->`->/<1000,300>[T\sP\sP X`\sP\sP T X`T\sP X`\sP T X;`T\sfs_X`\sfs_{T X}`\lam_X]
\morphism(0,300)<500,0>[T\sP\sP X`\sP T\sP X;\lam_{\sP X}]
\morphism(500,300)<500,0>[\sP T\sP X`\sP\sP T X;(\lam_X)_!]
\place(500,130)[\geq]
\efig$ & (3) & $\hT\psi\circ\hT\phi\leq\hT(\psi\circ\phi)$\\
& & (3') & $(\hT\phi)^{\odot}\cdot \overleftarrow{\hT\varepsilon_Y}
\leq \overleftarrow{\hT\varepsilon_X}
\cdot T\phi^{\odot}$
\end{tabular}

\begin{tabular}{p{0.5cm}p{6cm}p{0.3cm}p{5.2cm}}
(d) & \hskip 0.8cm $\bfig
\Atriangle<300,300>[\sP X`T\sP X`\sP TX;e_{\sP X}`(e_X)_!`\lam_X]
\place(300,120)[\geq] \efig$ & (4) & $\phi\circ e_X^{\circ}\leq e_Y^{\circ}\circ\hT\phi$
\end{tabular}

\begin{tabular}{p{0.5cm}p{6cm}p{0.3cm}p{5.2cm}}
(e) & $\bfig
\square/`->`->`->/<1000,300>[TT\sP X`\sP TT X`T\sP X`\sP T X;`m_{\sP X}`(m_X)_!`\lam_X]
\morphism(0,300)<500,0>[TT\sP X`T\sP T X;T\lam_X]
\morphism(500,300)<500,0>[T\sP T X`\sP TTX;\lam_{T X}]
\place(500,130)[\geq]
\efig$ & (5) & $\hT\hT\phi\circ m_X^{\circ}\le m_Y^{\circ}\circ\hT\phi$
\end{tabular}

   \begin{prop}\label{morecorr}
  Let $\lambda\in (\bbT,\CQ)\text{-}\rm{DIS}$ and $\hT\in (\bbT,\CQ)\text{-}\rm{EXT}$ be related under the
   correspondence of Proposition {\rm{\ref{correspondence}}}, so that $\hT = \Phi(\lambda), \lambda = \Psi(\hT)$.
   Then:
   $$\rm{(a)}\Leftrightarrow (1),\;\rm{(b)}\Leftrightarrow(2)\Leftrightarrow(2'),\;\rm{(a)\&(c)}\Rightarrow(3)\Leftrightarrow{(3')}\Rightarrow(c),\;(2')\&(3)\Rightarrow(a),\;\rm{(d)}\Leftrightarrow(4),\;\rm{(e)}\Leftrightarrow(5),$$
   and in each of these implications or equivalences one may replace the inequality sign by an equality sign on both sides of the implication or equivalence sign.
 \end{prop}

 \begin{proof}
 (a)$\Rightarrow$(1): The hypothesis (a) and Remark \ref{rules} give
 $$\overleftarrow{\hT\psi\circ(Tg)^{\circ}}=(Tg)_!\cdot\overleftarrow{\hT\psi}=(Tg)_!\cdot\lam_X\cdot T\overleftarrow{\psi}\leq\lam_Z\cdot T(g_!\cdot\overleftarrow{\psi})=\lam_Z\cdot T((g^{\circ})^{\odot}\cdot\overleftarrow{\psi})=\overleftarrow{\hT(\psi\circ g^{\circ})},$$
with equality holding when equality holds in (a).

(1)$\Rightarrow$(a): The hypotheses (0), (1), the naturality of $\varepsilon$ and the repeated application of Remark \ref{rules} give the inequality (a), with equality holding when equality holds in (1):
\begin{align*}
(Tf)_!\cdot\lambda_X&=((Tf)^{\circ})^{\odot}\cdot\overleftarrow{\hT\varepsilon_X}=\overleftarrow{\hT\varepsilon_X\cdot(Tf)^{\circ}}\\
&\leq\overleftarrow{\hT(\varepsilon_X\circ f^{\circ})}=
\overleftarrow{\hT((f_!)^{\circ}\circ\varepsilon_Y}=\overleftarrow{(T(f_!))^{\circ}\circ\hT\varepsilon_Y}\\
& =\overleftarrow{\hT\varepsilon_Y}\cdot T(f_!)=\lam_Y\cdot T(f_!).
\end{align*}

(b)$\Rightarrow(2')$:
$\quad\quad\; \overleftarrow{\hT(f^{\circ})}=\lam_Y\cdot T\overleftarrow{f^{\circ}}=\lam_Y\cdot T\sy_Y\cdot Tf\geq\sy_{TY}\cdot Tf=\overleftarrow{(Tf)^{\circ}}.$

$(2')\Rightarrow$(2)$\Rightarrow$(b): Consider $f=1_X$ and use the same steps as in (b)$\Rightarrow(2')$.
Trivially then, equality holds in (b) if, and only if, equality holds in (2), or $(2')$.

\quad

(a)\&(c)$\Rightarrow(3')$: With $\lambda:=\Psi (\hT)$, inequality $(3')$ follows from (a) and (c) and Remark \ref{rules}, with equality holding if it holds in both (a) and (c), as follows:
\begin{align*}
\lam_X\cdot T(\phi^{\odot})&=\lam_X\cdot T\sfs_X \cdot T(\phila_!)\\
&\geq\sfs_{TX}\cdot (\lam_X)_!\cdot \lam_{\sP Y}\cdot T(\phila_!)\\
&\geq\sfs_{TX}\cdot(\lam_X)_!\cdot(T\phila)_!\cdot\lam_Y\\
&=\sfs_{TX}\cdot(\overleftarrow{\hT\phi})_!\cdot\lam_Y=(\hT\phi)^{\odot}\cdot\lam_Y.
\end{align*}

$(3')\Rightarrow$(c): Inequality (c) follows when one puts $\phi=\varepsilon_X$ in $(3')$, with equality holding when it holds in $(3')$:
$$\lam_X\cdot T\sfs_X=\lam_X\cdot T(\varepsilon^{\odot}_X)\geq(\hT\varepsilon_X)^{\odot}\cdot\lam_{\sP X}
=\sfs_{TX}\cdot(\overleftarrow{\hT\varepsilon_X})_!\cdot\lam_{\sP X}
=\sfs_{TX}\cdot(\lam_X)_!\cdot\lam_{\sP X}.$$

$(3')\Rightarrow$(3): With $\lam_X=\overleftarrow{\hT\varepsilon_X}$ one obtains (3) from (3') and Remark \ref{rules}, as follows:
$$\overleftarrow{\hT\psi\circ\hT\phi}=(\hT\phi)^{\odot}\cdot\overleftarrow{\hT\psi}
=(\hT\phi)^{\odot}\cdot\lam_Y\cdot T\overleftarrow{\psi}\leq\lam_X\cdot T\phi^{\odot}\cdot T\overleftarrow{\psi}=\lam_X\cdot T(\overleftarrow{\psi\circ\phi})=\overleftarrow{\hT(\psi\circ\phi)}.$$

(3)$\Rightarrow(3')$: One exploits the naturality of $\varepsilon$ and (3) (putting $\psi=\varepsilon_Y)$ to obtain:
$$(\hT\phi)^{\odot}\cdot\overleftarrow{\hT\varepsilon_Y}=\overleftarrow{\hT\varepsilon_Y\circ\hT\phi}
\leq\overleftarrow{\hT(\varepsilon_Y\circ\phi)}
=\overleftarrow{\hT((\phi^{\odot})^{\circ}\circ\varepsilon_X)}
=\overleftarrow{(T\phi^{\odot})^{\circ}\circ\hT\varepsilon_X}
=\overleftarrow{\hT\varepsilon_X}\cdot T\phi^{\odot},$$

with equality holding precisely when equality holds in (3).

\quad

$(2')$\&(3)$\Rightarrow$(1): $\hT\phi\circ(Tg)^{\circ}\leq\hT\phi\circ\hT(g^{\circ})\leq\hT(\phi\circ g^{\circ}).$

\quad

(d)$\iff$(4): We show $``\Rightarrow"$; the implication $``\Leftarrow"$ follows similarly, with $\phi=\varepsilon_X$:
$$\overleftarrow{\phi\circ e_X^{\circ}}=(e_X)_!\cdot\phila\leq\lam_X\cdot e_{\sP X}\cdot\phila=
\lam_X\cdot T\phila\cdot e_Y=\overleftarrow{\hT\phi}\cdot e_Y=(\hT\phi)^{\odot}\cdot\overleftarrow{e_Y^{\circ}}
=\overleftarrow{e_Y^{\circ}\circ\hT\phi}.$$

(e)$\iff$(5): Since again $``\Leftarrow"$ follows by putting $\phi=\varepsilon_X$, we show only $``\Rightarrow"$:
\begin{align*}
\overleftarrow{\hT\hT\phi\circ m_X^{\circ}}&
=(m_X)_!\cdot\overleftarrow{\hT\hT\phi}=(m_X)_!\cdot\lam_{TX}\cdot T\overleftarrow{(\hT\phi)}=(m_X)_!\cdot\lam_{TX}\cdot T\lam_X\cdot TT\phila\\
&\leq\lam_X\cdot m_{\sP X}\cdot TT\phila=\lam_X\cdot T\phila\cdot m_Y=\overleftarrow{\hT\phi}\cdot m_Y=(\hT\phi)^{\odot}\cdot\overleftarrow{m_Y^{\circ}}=\overleftarrow{m_Y^{\circ}\circ\hT\phi}.
\end{align*}
\end{proof}

  A {\em lax extension} $\hT$ of the monad $\bbT$ to $\CQ\text{-}\bf{Rel}$ is
 a monotone $(\bbT,\CQ)$-extension family satisfying conditions (0), (2)-(5) for all $\phi:X\rto Y, \psi:Y\rto Z$ in $\CQ\text{-}\bf{Rel}$ and $h:Z\to Y$ in $\Set/\CQ_0$, {\em i.e.}, a left-whiskering lax functor $\hT:\CQ\text{-}\bf{Rel}\to\CQ\text{-}\bf{Rel}$ that coincides with $T$ on objects and makes $e^{\circ}:\hT\rto 1_{\Set/\CQ_0}$ and $m^{\circ}:\hT\rto\hT\hT$ lax natural transformations. We have proved in Propositions \ref{corr}, \ref{extfamrules} and \ref{morecorr} the following theorem (which corrects and generalizes Exercise III.1.I in \cite{MonTop}):

 \begin{thm}\label{correspondence}
 There is a bijective correspondence between the monotone distributive laws of the monad $\bbT$ over $\bbP_{\CQ}$ and the lax extensions $\hat{T}$ of $\bbT$ to $\CQ\text{-}\bf{Rel}$. These lax extensions are equivalently described as monotone $(\bbT,\CQ)$-extension families $\hT$ satisfying the following inequalities (for all $f,\phi,\psi$ as above):
 \begin{itemize}
 \item[\rm{1.}] $(Tf)_{\circ}\leq \hT(f_{\circ}),$
 \item[\rm{2.}] $(Tf)^{\circ}\leq \hT(f^{\circ}),$
 \item[\rm{3.}]$(=\rm{(3)})\;\; \hT\psi\circ\hT\phi\leq\hT(\psi\circ\phi),$
 \item[\rm{4.}] $(e_Y)_{\circ}\circ\phi\leq\hT\phi\circ (e_X)_{\circ},$
 \item[\rm{5.}] $(m_Y)_{\circ}\circ\hT\hT\phi\leq\hT\phi\circ(m_X)_{\circ}.$
 \end{itemize}
 \end{thm}

 For a lax extension $\hT$ of the monad $\bbT$ to $\CQ\text{-}\bf{Rel}$ we can now define:

 \begin{defn}
 A $({\mathbb T},\CQ)$-{\em category} $(X,\alpha)$ is a set $X$ over $\CQ_0$ equipped with a $\CQ$-relation $\alpha: X\rto TX$ satisfying the lax unit and multiplication laws
 $$1_X^{\circ}\leq e_X^{\circ}\circ\alpha\text{     and     }\hT\alpha\circ\alpha\leq m_X^{\circ}\circ\alpha.$$
 A $({\mathbb T},\CQ)$-{\em functor} $f:(X,\alpha)\to (Y,\beta)$ must satisfy
 $$\alpha\circ f^{\circ}\leq (Tf)^{\circ}\circ\beta.$$
 Hence, the structure of a $({\mathbb T},\CQ)$-category $(X,\alpha)$ consists of a family of $\CQ$-morphisms $\alpha(x,\Fx): |x|_X\to |\Fx|_{TX} \; (x\in X, \Fx\in TX)$, subject to the conditions
 $$ 1_{|x|}\leq\alpha(x,e_Xx)\text{    and    }\hT\alpha(\Fy,\FZ)\circ\alpha(x,\Fy)\leq\alpha(x,m_X\FZ),$$
 for all $x\in X, \Fy\in TX, \FZ\in TTX$. The $({\mathbb T},\CQ)$-functoriality condition for $f$ reads in pointwise form as
 $$\alpha(x,\Fy)\leq\beta(fx,Tf(\Fy))$$
 for all $x\in X, \Fy\in TX$. The emerging category is denoted by
 $$({\mathbb T},\CQ)\text{-}{\bf Cat};$$
 only if there is the danger of ambiguity will we write $({\mathbb T},\hat{T},\CQ)\text{-}{\bf Cat}$ to stress the dependency on the chosen extension $\hat{T}$.
 \end{defn}
 \begin{rem}\label{comparisonMonTop}
 When $\CQ$ is a commutative quantale $\sV$, then the structure of  a $({\mathbb T},\sV)$-category $(X,\alpha)$ may be written equivalently as a $\sV$-relation $TX\rto X$,
 and the notion takes on the familiar meaning (as presented in \cite{MonTop}). However, it is important to note that, because of the switch in direction of the $\sV$-relation $\alpha:X\rto TX$ (as a lax coalgebra structure) to a lax algebra structure $TX\rto X$ as in \cite{MonTop}, $({\mathbb T},\hat{T},\sV)\text{-}{\bf Cat}$ defined here actually becomes
 $({\mathbb T},\sV,\check{T})\text{-}{\bf Cat}$ as defined in \cite{MonTop}, III.1, with $\check{T}\varphi:=(\hat{T}(\varphi^{\circ}))^{\circ}$ and $\varphi^{\circ}:Y\rto X, \;\varphi^{\circ}(y,x)=\varphi(x,y)$,
 for all $\sV$-relations $\varphi:X\rto Y,\;x\in X,\;y\in Y$ (see Exercise III.1.J in \cite{MonTop}).
 \end{rem}

 Before presenting further examples, let us point out that $({\mathbb T},\CQ)$-categories and -functors are just disguised lax $\lam$-algebras with their lax homomorphisms, since $\CQ$-relations $\alpha:X\rto TX$ correspond bijectively to $\Set/\CQ_0$-morphisms $p:TX\to\sP X$ under the adjunction of Section 3.

  \begin{prop}\label{modelequivalence}
When $\lam$ and ${\mathbb T},\hat{T}$ are related by the correspondence of {\rm{Theorem} \ref{correspondence}}, then there is a (natural) isomorphism $$(\lam,\CQ)\text{-}{\bf Alg}\cong(\bbT,\hat{T},\CQ)\text{-}{\bf Cat}$$
 of categories which commutes with the underlying $\Set/\CQ_0$-functor.
 \end{prop}

 \begin{proof}
 Given a $({\mathbb T},\CQ)$-category structure $\alpha$ on $X$, repeated applications of the rules of Remark \ref{rules} confirm that $\overleftarrow{\alpha}$ makes $X$ a lax $\lam$-algebra:
 $$\overleftarrow{\alpha}\cdot e_X=\overleftarrow{e_X^{\circ}\circ\alpha}\geq \overleftarrow{1_X^{\circ}}=\sy_X,$$
 $$\overleftarrow{\alpha}\cdot m_X=\overleftarrow{m_X^{\circ}\circ\alpha} \geq\overleftarrow{\hT\alpha\circ\alpha}=\alpha^{\odot}\cdot\overleftarrow{\hT\alpha}=\sfs_X\cdot\overleftarrow{\alpha}_!\cdot\lam_X\cdot T\overleftarrow{\alpha}.$$
 Conversely, given a lax $\lam$-algebra structure $p$ on $X$, putting $\alpha:=p^{\circ}\cdot\varepsilon_X$ one has $\overleftarrow{\alpha}=p$, and the same computational steps as above show $\overleftarrow{e_X^{\circ}\cdot\alpha}\geq\overleftarrow{1_X^{\circ}}$ and $\overleftarrow{m_x^{\circ}\cdot\alpha}\geq\overleftarrow{\hT\alpha\cdot\alpha}$, so that
  $\alpha$ is a $({\mathbb T},\CQ)$-category structure on $X$.

  A $({\mathbb T},\CQ)$-functor $f:(X,\alpha)\to(Y,\beta)$ gives a lax $\lam$-homomorphism $f:(X,\overleftarrow{\alpha})\to(Y,\overleftarrow{\beta})$, since
  $$f_!\cdot\overleftarrow{\alpha}=\overleftarrow{\alpha\circ f^{\circ}}\leq\overleftarrow{(Tf)^{\circ}\circ\beta}=\overleftarrow{\beta}\cdot Tf,$$
  and conversely.
 \end{proof}

 \begin{exmp}\label{laxextexamples}
 \begin{itemize}
 \item[\rm{(1)}] Let $\bbT$ be a $\Set$-monad with a lax extension $\tilde{T}$ to $\Rel={\sf 2}\text{-}\Rel$ that we now wish to extend further to $\sD{\sf 2}\text{-}\Rel$.
 As in Proposition \ref{strict lifting}, we first consider a $\bbT$-algebra structure $\zeta:T{\sf 2}\to{\sf 2}$, which then allows us to consider
  $\bbT$ as a monad on $\Set/{\sf 2}$, the category of sets $X$ with a given subset $A$ (see Example \ref{secondexamples}(2)). Of course, one now wishes
 to compute $T(X,A)$ as the pair $(TX,TA)$. Since the array function of $X$ is the characteristic function $c_A$ of $A$, this is possible precisely when the $\Set$-functor $T$ satisfies the pullback transformation condition
 $$\bfig
\square<300,300>[A`{\sf 1}`X`{\sf 2};``\top`c_A]
\place(70,230)[\lrcorner] \place(700,150)[\Lra] \place(1170,230)[\lrcorner] \place(60,230)[\cdot] \place(1160,230)[\cdot]
\square(1100,0)/->`->``->/<400,300>[T A`T{\sf 1}`T X`T{\sf 2};```Tc_A]
\square(1500,0)/->``->`->/<350,300>[T{\sf 1}`{\sf 1}`T{\sf 2}`{\sf 2},;``\top`\zeta]
\efig$$
 and this condition certainly holds when $T$ is taut ({\em i.e.,} preserves pullbacks of monomorphisms) and
 $\zeta^{-1}{\sf 1}=T{\sf 1}$.
 Since a morphism
 $\phi:(X,A)\rto(Y,B)$ (where $A\subseteq X, B\subseteq Y)$ in $\sD{\sf 2}\text{-}\Rel$ is completely determined by the restricted relation $\phi_{\text{rest}}: A\rto B$, one may now declare $\Fx$ to be $\hat{T}\phi$-related to $\Fy$ if, and only if, $\Fx\in TA, \Fy\in TB$ and $\Fx$ is $\tilde{T}\phi$-related to $\Fy$, to obtain a lax extension of $\bbT$ to
  $\sD{\sf 2}\text{-}\Rel$.

With $\tilde{T}$ and the $\bbT$-algebra structure $\zeta$ on $\sf 2$ given such that $\zeta^{-1}{\sf 1}=T{\sf 1}$, the objects $(X,A, \alpha)$ of the category $(\bbT,\hat{T},\sD{\sf 2})\text{-}\Cat$
may be described as sets $X$ with a subset $A$ such that $(A,\alpha)\in (\bbT,\tilde{T},{\sf 2})\text{-}\Cat$; morphisms $f:(X,A,\alpha)\to(Y,B,\beta)$
are maps $f:X\to Y$ with $f^{-1}B=A$ whose restrictions $A\to B$ are $(\bbT,{\sf 2})$-functors. The list monad
$\mathbb L$ (with $\zeta:L{\sf 2}\to{\sf 2}$ given by $\wedge$) and the ultrafilter monad $\mathbb U$
both satisfy our hypotheses, and  $(\bbT,\hat{T},\sD{\sf 2})\text{-}\Cat$ then describes the categories of {\bf ParMulOrd} and {\bf ParTop} of {\em partial multi-ordered sets} and {\em partial topological spaces}, respectively.

  \item[\rm{(2)}] Expanding on Example \ref{laxdistlawexamples}(2),(3) and Example \ref{lamalgebraexamples}(2), with $\mathbb L$ laxly extended to $\sD [0,\infty]\text{-}{\bf Rel}$,  one obtains as $({\mathbb L},\sD [0,\infty])\text{-}\Cat$ the category {\bf ParMulMet} of {\em partial multi-metric spaces} whose objects $X$ may be described as sets carrying a distance function $a:LX\times X\to [0,\infty]$ (see Remark \ref{comparisonMonTop}) subject to the conditions
\begin{align*}
&\text{max}(\displaystyle\sum\limits_{i=1}^n a(x_i,x_i),\; a(y,y))\leq a((x_1,...,x_n)y),\\
&a((\underbrace{x_{1,1},\dots,x_{1,n_1}}_{\Fx_1},\dots,\underbrace{x_{m,1},\dots,x_{m,n_m}}_{\Fx_m}),z)\leq\Big(\sum\limits_{i=1}^m a(\Fx_i,y_i)-a(y_i,y_i)\Big)+a((y_1,\dots,y_m),z);\\
\end{align*}
their morphisms $f:(X,a)\to(Y,b)$ must satisfy
$$ b((f(x_1),\dots,f(x_n)),f(y))\leq a((x_1,\dots,x_n),y)\text{   and   }
b(f(x),f(x))=a(x,x),$$
for all $x,x_1,...,x_n,y\in X$.
 \end{itemize}
 \end{exmp}

 \section{Algebraic functors, change-of-base functors}
 Here we consider the standard types of functors arising from a variation in the two parameters defining the categories $(\lam,\CQ)\text{-}{\bf Alg}\cong(\bbT,\hat{T},\CQ)\text{-}\Cat$, which have been discussed earlier, in the quantale-monad-enriched case (see \cite{CleHofTho, MonTop}) as well as for $\bbT$ in more general settings (see \cite{CarboniKellyWood}), but not in the current monad-quantaloid-enriched context, which does require some extra pecautions.

  Let us first consider two monads ${\mathbb T}=(T,m,e),\; {\mathbb S}=(S,n,d)$ on ${\bf Set}/\CQ_0$, both monotonely distributing over ${\mathbb P}_{\CQ}$, via the monotone distributive laws $\lam,\kappa$, respectively; equivalently, both coming equipped with lax extensions $\hat{T}$ and $\hat{S}$ to $\CQ\text{-}{\bf Rel}$, respectively.
 An {\em algebraic morphism} $h:(\bbT,\hat{T})\rto({\mathbb S},\hat{S})$ of lax extensions is a family of $\CQ$-relations $h_X:TX\rto SX\;(X\in \Set/\CQ_0)$, satisfying the following conditions for all $f:X\to Y\text{ in } \Set/\CQ_0,\;\phi:X\rto Y,\;\alpha:X\rto TX \text{ in }\CQ\text{-}{\bf Rel}$:

 \begin{tabular}{llr}
 &&\\
 a. & $h_X\circ(Tf)^{\circ}\leq (Sf)^{\circ}\circ h_Y$, &\quad\quad\quad\quad
 (lax naturality)\\
 b. & $e_X^{\circ}\leq d_X^{\circ}\circ h_X$,  &\quad (lax unit law)\\
 c. & $\hat{S}h_X\circ h_{TX}\circ m_X^{\circ}\leq n_X^{\circ}\circ h_X$, &\quad\quad\quad\quad\quad\quad\quad\; (lax multiplication law)\\
 d. & $\hat{S}\phi\circ h_X\leq h_Y\circ\hat{T}\phi$, &\quad (lax compatability)\\
 e. & $\hat{S}(h_X\circ\alpha)\leq\hat{S}h_X\circ\hat{S}\alpha$. &\quad (strictness at $h$)\\
 &&\\
 \end{tabular}

Note that, because of the lax functoriality of $\hat{S}$, ``$\leq$" in condition e actually amounts to ``=". Putting now $\tau_X:=\overleftarrow{h_X}$ and exploiting Remark \ref{rules},  we may equivalently call a family of $\Set/\CQ_0$-morphisms $\tau_X:SX\to\sP_{\CQ}TX\;(X\in\Set/\CQ_0)$ an {\em algebraic morphism} $\tau:\kappa\to\lambda$ of monotone distributive laws if the following conditions hold for all $f:X\to Y, \;g:Y\to\sP_{\CQ}X,\; p: TX\to \sP_{\CQ}X\text{ in }\Set/\CQ_0$:

\begin{tabular}{llr}
&&\\
a'.& $(Tf)_!\cdot\tau_X\leq\tau_Y\cdot Sf,$&(lax naturality)\\
b' & $\sy_{TX}\cdot e_X\leq\tau_X\cdot d_X,$&(lax unit law)\\
c'.& $(m_X)_!\cdot\sfs_{TTX}\cdot(\tau_{TX})_!\cdot\kappa_{TX}\cdot S\tau_X\leq\tau_X\cdot n_X,$&(lax multiplication law)\\
d'.& $\sfs_{TX}\cdot(\tau_X)_!\cdot\kappa_X\cdot Sg\leq\sfs_{TX}\cdot(\lam_X)_!\cdot (Tg)_!\cdot\tau_Y,$&(lax compatibility)\\
e'.& $\kappa_X\cdot S\sfs_X\cdot S(p_!)\cdot S\tau_X\leq \sfs_{SX}\cdot(\kappa_X)_!\cdot(Sp)_!\cdot\kappa_{TX}\cdot S\tau_X.$&(strictness at $p$)\\
&&\\
\end{tabular}

A routine calculation shows:
\begin{prop}\label{algfun}
Every algebraic morphism $h:(\bbT,\hat{T})\to({\mathbb S},\hat{S})$ of lax extensions induces the {\rm algebraic functor}
$$A_h:(\bbT,\hat{T},\CQ)\text{-}{\bf Cat}\to({\mathbb S},\hat{S},\CQ)\text{-}\Cat,\;(X,\alpha)\mapsto (X,h_X\circ\alpha).$$
When $h$ is equivalently described as an algebraic morphism $\tau:\kappa\to\lambda$, then $A_h$ is equivalently described as the algebraic functor
$$A_{\tau}:(\lam,\CQ)\text{-}{\bf Alg}\to(\kappa,\CQ)\text{-}{\bf Alg}, \;(X,p)\mapsto (X,\nu_X\cdot p_!\cdot\tau_X).$$
\end{prop}

Considering $\mathbb S$ and $\hat{S}$ identical or, equivalently, $\kappa=1_{\sP}$,  with the algebraic morphism
$h_X=e_X^{\circ}$ or, equivalently, $\tau_X=\sy_{TX}\cdot e_X$, one obtains:

\begin{cor}\label{algfuncor}
For every monad $\mathbb T$ on $\Set/\CQ_0$ with lax extension $\hat{T}$ and corresponding monotone distributive law $\lam$, there is an algebraic functor
$$A:(\bbT,\CQ)\text{-}\Cat\to\CQ\text{-}\Cat, (X,\alpha)\mapsto(X,e_X^{\circ}\circ\alpha)$$
that is equivalently described by
$$A:(\lam,\CQ)\text{-}{\bf Alg}\to\CQ\text{-}\Cat, (X,p)\mapsto (X,p\cdot e_X).$$
\end{cor}

\begin{exmp}\label{algfctrexample}
(See \cite{LaiTholen}.) For the powerset monad ${\mathbb P}={\mathbb P}_{\sf 2}$ and the ultrafilter monad $\mathbb U$ with their monotone distributive laws $\delta$ and $\beta$ over ${\mathbb P}_{\sV}$ of Example \ref{laxdistlawexamples}(4),(5) and their corresponding lax extensions
$\hat{\sP}$ and $\overline{U}$ to $\sV\text{-}{\bf Rel}$, where $\sV=(\sV,\otimes,{\sf k})$ is a commutative and completely distributive quantale, the algebraic morphism
$h$ with $h_X:UX\rto\sP X,\;\;h_X(\Fx,A)={\sf k}$ if $A\in\Fx\in UX$, and $h_X(\Fx,A)=\bot$ else, induces the algebraic functor
$$({\mathbb U},\sV)\text{-}\Cat\to({\mathbb P},\sV)\text{-}\Cat\cong\sV\text{-}{\bf Cls},$$
which actually takes values in $\sV\text{-}{\bf App}$ and facilitates the isomorphism of categories $({\mathbb U},\sV)\text{-}\Cat\cong\sV\text{-}{\bf App}$
already mentioned in equivalent form in Example \ref{lamalgebraexamples}(4).
\end{exmp}

In order to describe change-of-base functors in the general setting of this paper, let us now consider a lax homomorphism $\vartheta:\CQ\to\CR$ of quantaloids, so that we have a lax natural transformation
$\vartheta:B_{\vartheta_0}\sP_{\CQ}\to\sP_{\CR} B_{\vartheta_0}$ (see the end of Section 3), and a $\Set$-monad $\bbT=(T,m,e)$ which, according to Proposition \ref{strict lifting}, has been lifted to $\Set/\CQ_0$ and $\Set/\CR_0$ via $\bbT$-algebra structures $\zeta:T\CQ_0\to\CQ_0$ and $\eta:T\CR_0\to\CR_0$, respectively, such that
$\vartheta_0:\CQ_0\to\CR_0$ is a $\bbT$-homomorphism. The liftings of $\bbT$ to $\Set/\CQ_0$ and $\Set/\CR_0$ commute with the ``discrete change-of-base functor" $B_{\vartheta_0}$, that is:
$B_{\vartheta_0}T=TB_{\vartheta_0},B_{\vartheta_0}e=eB_{\vartheta_0}, B_{\vartheta_0}m=mB_{\vartheta_0}$. (These provisions are, of course, trivially satisfied when $\CQ$ and $\CR$ are quantales.)

Extendinging now $B_{\vartheta_0}$ to a functor $\tilde{B_{\vartheta}}:\CQ\text{-}{\bf Rel}\to\CR\text{-}{\bf Rel}$ by
$(\tilde{B_{\vartheta}}\phi)(x,y)=\vartheta(\phi(x,y))$  and considering lax extensions $\hat{T}, \check{T}$ of $\bbT$ to $\CQ\text{-}{\bf Rel}, \CR\text{-}{\bf Rel}$, respectively, we call $\vartheta$
 {\em compatible} with $\hat{T},\check{T}$ if
 $$\check{T}\tilde{B_{\vartheta}}\phi\leq \tilde{B_{\vartheta}}\hat{T}\varphi\quad\quad\quad\quad(\star)$$
 for all $\phi:X\rto Y $ in $\CQ\text{-}{\bf Rel}$. (Note that the two $\CR$-relations in $(\star)$ are comparable since $B_{\vartheta_0}T=TB_{\vartheta_0}$.) If we describe the two lax extensions $\hat{T},\check{T}$ equivalently by the monotone distributive laws $\lam,\kappa$, respectively, using the natural lax natural transformation
 $\vartheta: B_{\vartheta_0}\sP_{\CQ}\to\sP_{\CR} B_{\vartheta_0}$ (see the end of Section 3) and the easily
 verified rule $\overleftarrow{\tilde{B_{\vartheta}}\phi}=\vartheta_X\cdot B_{\vartheta_0}\overleftarrow{\phi}$,
 we see that $(\star)$ may equivalently be formulated as
 $$\kappa B_{\vartheta_0}\cdot T\vartheta\leq\vartheta T\cdot B_{\vartheta_0}\lam\quad\quad(\star\star).$$

 Now we can state the following Proposition, which one may prove using lax extensions and  transcribing the known proof for the quantale case (see \cite{MonTop}, III.3.5); alternatively, one may proceed by using the monotone distributive laws and the lax monad inequalities of $\vartheta$ as stated at the end of Section 3.

\begin{prop}\label{changeofbase}
Under hypothesis $(\star)$ one obtains the $\rm {change\text{-}of\text{-}base\text{ }functor}$
$$B_{\vartheta}:(\bbT,\hat{T},\CQ)\text{-}{\bf Cat}\to(\bbT,\check{T},\CR)\text{-}{\bf Cat},\;(X,\alpha)\mapsto (B_{\vartheta_0}X,\tilde{B_{\vartheta}}\alpha).$$
Under hypothesis $(\star\star)$ this functor is equivalently described as
 $$B_{\vartheta}:(\lam,\CQ)\text{-}{\bf Alg}\to(\kappa,{\CR})\text{-}{\bf Alg}, \;(X,p)\mapsto (B_{\vartheta_0}X,\vartheta_X\cdot B_{\vartheta_0}p).$$
\end{prop}

\begin{exmp}\label{lastexample} For a commutative and (for simplicity) divisible quantale $\sV$, we consider the lax extensions of the list monad $\mathbb L$ to $\sV\text{-}{\bf Rel}$ and ${\sf DV}\text{-}{\bf Rel}$ induced by the monotone distributive laws of Example \ref{laxdistlawexamples}(2),(3), which we may both denote by $\hat{L}$.
In fact, for $\phi:X\rto Y$ and $x_i\in X, y_j\in Y$ one has
$$\hat L\phi((x_1,...,x_n),(y_1,...,y_m))=\phi(x_1,y_1)\otimes...\otimes\phi(x_m,y_m)\text{  if }m=n,$$
to be interpreted as an arrow $|x_1|\otimes...\otimes|x_n|\to/~>/ |y_1|\otimes...\otimes |y_n|$ in the $\sf DV$-case, and the value is $\bot$ otherwise. For the homomorphism $\iota:\sV\to{\sf DV}$ and its retractions $\delta, \gamma$ as described in Section 2, one sees that
 ${\tilde B}_{\iota}$ embeds $\sV\text{-}{\bf Rel}$ fully into ${\sf DV}\text{-}{\bf Rel}$, providing every set with the constant array function with value $\sf k$, while its retractions ${\tilde B}_{\delta}$ and ${\tilde B}_{\gamma}$ are given by
 ${\tilde B}_{\delta}\phi(x,y)=|y|\rda\phi(x,y)$ and ${\tilde B}_{\gamma}\phi(x,y)=\phi(x,y)\lda|x|.$
 Since the compatability condition $(\star)$ holds for all, $\iota, \delta$ and $\gamma$ (strictly so for $\iota$), as ``liftings" of the corresponding functors mentioned in Section 2, one obtains the full embedding
 $B_{\iota}:({\mathbb L},\sV)\text{-}\Cat\to({\mathbb L}, {\sf DV})\text{-}\Cat$
 and its retractions $B_{\delta},B_{\gamma}$, which we describe explicitly here only in the case $\sV=[0,\infty]$ using the notation of Example \ref{laxextexamples}(2):
\begin{align*}
&B_{\delta},B_{\gamma}:{\bf ParMultMet}\to{\bf MulMet}\\
&\quad\;\;\, B_{\delta}:(X,a)\mapsto (X,a_{\delta}),\;a_{\delta}((x_1,...,x_n),y)=a((x_1,...,x_n),y)-\sum\limits_{i=1}^n a(x_i,x_i),\\
&\quad\;\;\, B_{\gamma}:(X,a)\mapsto (X,a_{\gamma}),\;a_{\gamma}((x_1,...,x_n),y)=a((x_1,...,x_n),y)-a(y,y).\\\end{align*}

The full reflective embedding
$E_{\sf V}:{\sf DV}\text{-}{\bf Cat}\to{\sf V}\text{-}{\bf Cat}/{\sf V}$
of Section 2 may be ``lifted" along the algebraic functors $({\mathbb L},{\sf DV})\text{-}\Cat\to {\sf DV}\text{-}\Cat$
and $({\mathbb L},\sV)\text{-}\Cat/\sV\to\sV\text{-}\Cat/\sV$ to obtain a full reflective embedding
$$E=E_{{\mathbb L},\sf V}:({\mathbb L},{\sf DV})\text{-}{\bf Cat}\to({\mathbb L},{\sf V})\text{-}{\bf Cat}/{\sf V},$$
which we briefly describe next, always assuming that $\sV$ be commutative and divisible.
First, in accordance with the general setting of  III.5.3 of \cite{MonTop}, we combine the monoid structure of $\sV$ with its internal hom and regard $\sV$ as an $({\mathbb L}, \sV)$-category $(\sV,h)$ with $h:L\sV\rto \sV$ (see Remark \ref{comparisonMonTop}) given by $h((v_1,...,v_n),u)=(v_1\otimes...\otimes v_n)\lda u$.
Now $E$ provides an $({\mathbb L}, {\sf DV})$-category $(X,a)$ with the $({\mathbb L},{\sf V})$-category structure
 $d$ defined by $d((x_1,...,x_n),y)=a((x_1,...,x_n),,y)\lda a(y,y)$ and considers it an
 $({\mathbb L},{\sf V})$-category over {\sf V} via $tx=a(x,x)$. Conversely, the reflector provides an
$({\mathbb L},{\sf V})$-category $(X,d)$ that comes equipped with an $({\mathbb L},{\sf V})$-functor $t:X\to {\sf V}$, with the $({\mathbb L}, {\sf DV})$-category structure $a$ defined by $a((x_1,...,x_n),y)=d((x_1,...,x_n),y)\otimes ty$.

 In the case $\sV=[0,\infty]$ the functor $E$ becomes an isomorphism of categories, so that in the notation of
 Example \ref{laxextexamples}(2) one has
 $${\bf ParMulMet}\cong{\bf MulMet}/[0,\infty].$$
 Therefore, just as described in Section 2 in the ``non-multi" case, the standard construction of a right adjoint to the functor $\Sigma: {\bf MulMet}/[0,\infty]\to{\bf MulMet} $ therefore gives a right adjoint to $B_{\gamma}: {\bf ParMulMet}\to{\bf MulMet}$.
 \end{exmp}

 \section{Comparison with Hofmann's topological theories}

 In \cite{Hofmann}, for a $\Set$-monad ${\mathbb T}=(T,m,e)$ and a {\em commutative} quantale $\sV=((\sV,\otimes,{\sf k})$, Hofmann considers maps $\xi:T\sV\to\sV$ satisfying the following conditions:

\begin{tabular}{lr}
&\\
1. $1_{\sV}\leq\xi\cdot e_{\sV}, \quad\xi\cdot T\xi\leq\xi\cdot m_{\sV}$ & (lax $\mathbb T$-algebra laws);\\
2*. ${\sf k}\cdot \zeta\leq \xi \cdot T{\sf k},\quad      \otimes\cdot(\xi\times\xi)\cdot \text{can}\leq\xi\cdot T(\otimes)$& (lax $\mathbb T$-homom. laws);\\
3. $\forall f, g: Y\to\sV\text{ in  }\Set\quad(f \leq g\Rightarrow\xi\cdot Tf\leq\xi\cdot Tg)$ &(monotonicity);\\
4. $\xi_X(\sigma):=\xi\cdot T\sigma\quad(\sigma\in\sP_{\sV}X=\sV^X)$ defines a nat. transf. $\sP_{\sV}\to\sP_{\sV}T$&(naturality).\\
&\\
\end{tabular}
Here ${\sf k}$ and $\otimes$  are considered as maps $1\to\sV$ and $\sV\times\sV\to \sV$,
 respectively;
 $\text{can}:T(\sV\times \sV)\to T\sV\times T\sV$ is the canonical map with components $T\pi_1, T\pi_2$, where $ \pi_1,\pi_2 $ are product projections, and (in accordance with the notation introduced in Proposition \ref{strict lifting}) $\zeta:T1\to 1$ is the unique map onto a singleton set $1$. Note that Hofmann \cite{Hofmann} combined conditions 3 and 4 to a single axiom; however, the separation as given above (and in \cite{ClementinoTholen}) is easily seen to be equivalent with Hofmann's combined axiom and will make the comparison with the conditions of Definition \ref{toptheory} more transparent.

 Let us now compare these conditions with conditions 0--3 for a topological theory as given in Definition \ref{toptheory}, in the case that $\CQ=\sV$ is a commutative quantale.
 First we give a direct comparison of condition 2* with condition 2 of Definition \ref{toptheory} which, in the current context, reads as follows:

\begin{tabular}{lr}
&\\
$2.\;\; \sy_1\cdot\zeta\leq\xi\cdot T\sy_1, \;\; \sfs_1\cdot\theta\leq\xi\cdot T\sfs_1$ & \quad\quad\quad\quad\quad\quad\quad\quad\quad\quad\quad\quad\quad\quad\quad(lax ${\mathbb T}$ homom. laws);\\
&\\
\end{tabular}
here $\zeta:T1\to 1$  is trivial, and $\theta=\xi_!\cdot  (\zeta\cdot Tt)^!\cdot\xi\cdot T(t_!), \;\text{for } t:\sV\to 1$. Indeed, the latter condition implies the former, as we show first.

 \begin{prop}\label{ThoimpliesHof}
 Every map $\xi:T\sV\to\sV$ satisfying condition  {\rm 2} satisfies condition {\rm 2*}.
 \end{prop}

 \begin{proof}
 Since ${\sf k}=\sy_1$, the first inequality of condition 2 actually coincides with the first inequality of condition 2*.  The crucial ingredient to comparing the second inequalities in both conditions is the map

 $$ \chi:\sV\times\sV\to\sP_{\sV}\sV=\sV^{\sV},\quad\chi(u,v)(w)= u\otimes(\sy_{\sV}v)(w) =\left\{\begin{array}{ll}
u
 & \text{if}\ w=v,\\
\bot & \text{else}
\end{array}\right\},
 $$
since, as one easily verifies, $\sfs_1\cdot\chi=\otimes$. It now suffices to show

$$(*) \quad\quad\quad\quad\chi \cdot(\xi\times\xi)\cdot\text{can}\leq\theta\cdot T\chi;$$
indeed, one can then conclude from $\sfs_1\cdot\theta\leq\xi\cdot T\sfs_1$
the desired inequality, as follows:
$$\otimes\cdot(\xi\times\xi)\cdot\text{can}=\sfs_1\cdot\chi\cdot(\xi\times\xi)\cdot\text{can}\leq\sfs_1\cdot\theta\cdot T\chi\leq\xi\cdot T\sfs_1\cdot T\chi=\xi\cdot T(\otimes).$$
In order to check $(*)$, let $\Fw\in T(\sV\times\sV)$ and $z\in \sV$. On one hand, with $\Fx:=T\pi_1(\Fw), \Fy:=T\pi_2(\Fw)$, one obtains
$$(\chi\cdot(\xi\times\xi)\cdot\text{can}(\Fw))(z)=\chi(\xi(\Fx),\xi(\Fy))(z)=
\left\{\begin{array}{ll}
\xi(\Fx)
 & \text{if}\ z=\xi(\Fy),\\
\bot & \text{else}
\end{array}\right\},$$
and on the other, with $\Fz:= T\chi(\Fw)$, and since $t_!\cdot\chi=\pi_1$, one obtains

\begin{align*}
(\theta\cdot T\chi(\Fw))(z)&=(\xi_!\cdot(\zeta\cdot Tt)^!\cdot\xi\cdot
T(t_!)(\Fz))(z)\\
&=\bigvee_{\Fa\in T\sV, \xi(\Fa)=z}((\zeta\cdot Tt)^!\cdot\xi\cdot T(t_!)(\Fz))(z)\\
&=\bigvee_{\Fa\in T\sV, \xi(\Fa)=z}\xi(T(t_!)(\Fz))=\left\{\begin{array}{ll}
\xi(\Fx)
 & \text{if  }\exists\; \Fa\in T\sV\;(\xi(\Fa)=z)\\
\bot & \text{else}
\end{array}\right\},
\end{align*}
which shows $(*)$.
\end{proof}

Next we will show that, in the presence of conditions 1, 3, 4, conditions 2 and 2* become equivalent, provided that the $\Set$-functor $T$ of $\mathbb T$ satisfies the {\em Beck-Chevalley condition (BC)}, that is: if $T$ transforms (weak) pullback diagrams in $\Set$ into weak pullback diagrams (see \cite{Hofmann, MonTop}). Note that the $\Set$-functors of both $\mathbb L$ and $\mathbb U$ satisfy BC.

Calling a topological theory $\xi$ (as defined in Definition \ref{toptheory}) {\em natural} if $\xi $ satisfies condition 4 above, we can show:

\begin{thm}\label{comparison theorem}
For a commutative quantale $\sV$ and a $\Set$-monad $\mathbb T$ with $T$ satisfying the Beck-Chevalley condition, the natural topological theories for $\mathbb T$ and $\sV$ are characterized as the maps $\xi$ satisfying Hofmann's conditions {\rm 1, 2*, 3, 4}.
\end{thm}

\begin{proof}
From Proposition \ref{ThoimpliesHof} we know that every natural topological theory satisfies Hofmann's conditions. Conversely, having $\xi$ satisfying Hofmann's conditions, since $T$ satisfies BC, one can define the induced {\em lax Barr-Hofmann extension} $T_{\xi}$ of $\mathbb T$, as given in Definition 3.4 of  \cite{Hofmann}:
$$(T_{\xi}\varphi)(\Fx,\Fy)=\bigvee\{\xi\cdot(T|\varphi|)(\Fw)\;|\;\Fw\in T(X\times Y), T\pi_1(\Fw)=\Fx, T\pi_2(\Fw)=\Fy\},\quad\quad(\dagger)$$
for all $\sV$-relations $\varphi:X\rto Y, \;\Fx\in TX,\Fy\in TY$, with $|\varphi |:X\times Y\to \sV$ denoting the map
 giving the values of $\varphi$. Let $\lam:=\Psi(T_{\xi})$
 be the corresponding monotone distributive law (see Theorem \ref{correspondence}), and $\overline{\xi}=\xi^{\lam}$ the induced topological theory (see Proposition \ref{concentration1}), {\em i.e.},
$$\overline{\xi}=\zeta_!\cdot\lam_1=\zeta_!\cdot\overleftarrow{T_{\xi}\varepsilon_1},$$
with $\varepsilon_1:1\rto \sV$ the counit at $1$ of the discrete presheaf adjunction.
Since $|\varepsilon_1|:1\times\sV\to\sV$ and $\pi_2:1\times\sV\to\sV$ may both be identified with the identity map on $\sV$, this formula gives, for all $\Fa\in T\sV$,
$$\overline{\xi}(\Fa)=\bigvee\{(T_{\xi}\varepsilon_1)(\Fb,\Fa)\;|\; \Fb\in T1\}=\bigvee\{\xi(\Fw)\;|\;\Fw\in T\sV, T\pi_2(\Fw)=\Fa\}=\xi(\Fa),$$
 Consequently, since $\overline{\xi}$ is induced by a monotone distributive law, $\xi=\overline{\xi}$ is a topological theory, with naturality given by hypothesis.
\end{proof}
In \cite{ClementinoTholen} we showed that, when $T$ satisfies BC, the assignment $\xi\mapsto T_{\xi}$ of $(\dagger)$ defines a bijective correspondence between the maps $\xi$ satisfying conditions 1, 2*, 3, 4 and those lax extensions $\hat{T}$
of $\mathbb T$ that are
\begin{itemize}
\item {\em left-whiskering}, that is: $\hat{T}(g_{\circ}\circ\varphi)=(Tg)_{\circ}\circ\hat{T}\varphi$ for all $\sV$-relations
$\varphi:X\rto Y$ and maps $g:Y\to Z$; and
\item {\em algebraic}, that is:
$$\hat{T}\varphi(\Fx,\Fy)=\bigvee\{\hat{T}(\varphi^1)(\Fb,\Fw)\;|\;\Fb\in T1, \Fw\in T(X\times Y), T\pi_1(\Fw)=\Fx, T\pi_2(\Fw)=\Fy\},$$ for all $\sV$-relations $\varphi:X\rto Y$; here $\varphi^1$ has the same values as $\varphi$ but is considered as a $\sV$-relation $1\rto X\times Y$.
\end{itemize}
(The proof of this characterization is easily reconstructed by following the proof of Theorem \ref{minimal} below.) We therefore obtain with Theorem \ref{comparison theorem}:
\begin{cor}\label{CT correspondence}
Under the hypotheses of {\rm Theorem \ref{comparison theorem}}, the assignment $\xi\mapsto T_{\xi}$ of $(\dagger)$ defines a bijective correspondence between the natural topological theories for $\mathbb T$ and $\sV$ and the left-whiskering and algebraic lax extensions of $\mathbb T$ to $\sV\text{-}{\bf Rel}$.
\end{cor}

The following chart summarizes the correspondences described in this paper:

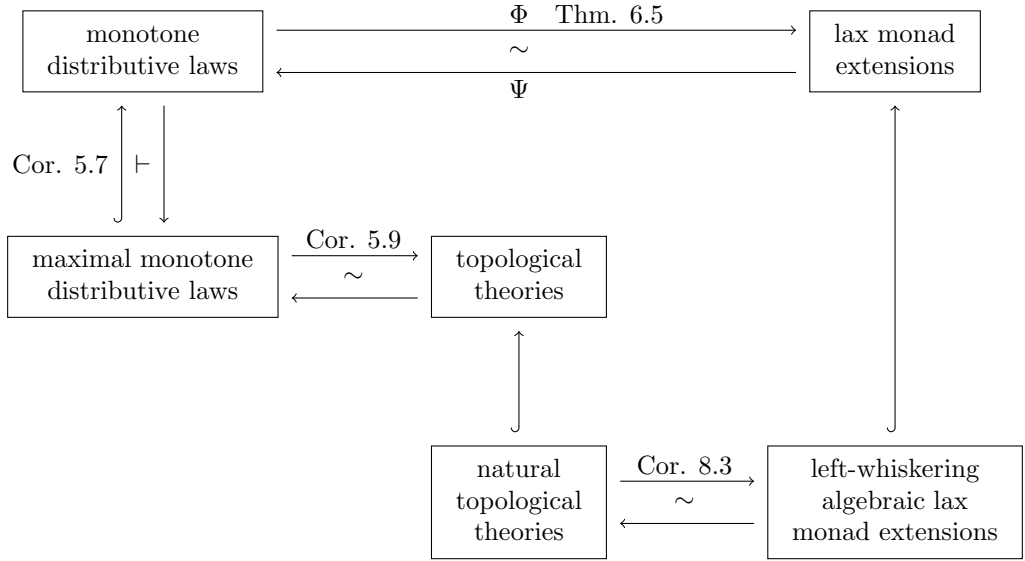
\begin{center}
\begin{tikzpicture}
\node[draw] (M) at (-5,0) {\begin{tabular}{c} monotone \\ distributive laws\end{tabular}};
\node[draw] (MM) at (-5,-3) {\begin{tabular}{c}maximal monotone\\ distributive laws\end{tabular}};
\node[draw] (T) at (0,-3) {\begin{tabular}{c}topological\\ theories\end{tabular}};
\node[draw] (NT) at (0,-6) {\begin{tabular}{c}natural \\ topological \\ theories\end{tabular}};
\node[draw] (W) at (5,-6) {\begin{tabular}{c}left-whiskering \\ algebraic lax \\ monad extensions\end{tabular}};
\node[draw] (L) at (5,0) {\begin{tabular}{c}lax monad \\ extensions\end{tabular}};
\draw[transform canvas={yshift=8pt},shorten >=5pt,shorten <=5pt,->] (M)--(L);
\draw[transform canvas={yshift=-8pt},shorten >=5pt,shorten <=5pt,->] (L)--(M);
\draw[transform canvas={xshift=-8pt},shorten >=5pt,shorten <=5pt,right hook->] (MM)--(M);
\draw[transform canvas={xshift=8pt},shorten >=5pt,shorten <=5pt,->] (M)--(MM);
\draw[transform canvas={yshift=8pt},shorten >=5pt,shorten <=5pt,->] (MM)--(T);
\draw[transform canvas={yshift=-8pt},shorten >=5pt,shorten <=5pt,->] (T)--(MM);
\draw[transform canvas={yshift=8pt},shorten >=5pt,shorten <=5pt,->] (NT)--(W);
\draw[transform canvas={yshift=-8pt},shorten >=5pt,shorten <=5pt,->] (W)--(NT);
\draw[shorten >=5pt,shorten <=5pt,right hook->] (NT)--(T);
\draw[shorten >=5pt,shorten <=5pt,right hook->] (W)--(L);
\node at (-5,-1.5) {$\vdash$};
\node at (-6.1,-1.5) {Cor. \ref{Galois}};
\node at (0,0) {$\sim$};
\node at (0,0.5) {$\Phi$};
\node at (0,-0.5) {$\Psi$};
\node at (1.2,0.5) {Thm. \ref{correspondence}};
\node at (-2.2,-3) {$\sim$};
\node at (-2.2,-2.5) {Cor. \ref{bijectiveGalois}};
\node at (2.2,-6) {$\sim$};
\node at (2.2,-5.5) {Cor. \ref{CT correspondence}};
\end{tikzpicture}
\end{center}

Caution is needed when reading this chart as a diagram, as it commutes only in a limited way. The following Remark and Theorem shed light on this cautionary note.
\begin{rem}\label{final remark}
\begin{itemize}
\item [\rm{(1)}] The Proof of Theorem \ref{comparison theorem} shows that, starting with a natural topological theory and, under the provisions of Theorem \ref{comparison theorem} on $\sV$ and $\bbT$,
 chasing it counterclockwise all around the chart, one arrives at the same topological theory.
\item [\rm{(2)}] However, under the assumptions of Theorem \ref{comparison theorem} on $\sV$ and $\mathbb T$, chasing a natural topological theory $\xi$ upwards on the two possible paths one obtains very different types of lax monad extensions; their typical properties appear to be almost disjoint. Most remarkably, assigning to $\xi$ the {\em maximal} monotone distributive law
 $\lam^{\xi}$ and then the lax monad extension $\hat{T}=\Phi(\lam^{\xi})$, one observes easily that, for $\varphi:X\rto Y,\;\Fx\in TX, \Fy\in TY$ and $a:X\to 1$,
$$\hat{T}\varphi(\Fx,\Fy)=\xi(T(a_!\cdot\overleftarrow{\varphi})(\Fy))$$
does not depend on $\Fx\,$! But also the other path up $(\xi\mapsto T_{\xi})$ leads to quite special monad extensions, since being left-whiskering and algebraic are restrictive properties which, for example, exclude all extensions $\hat{T}$ that fail to satisfy the symmetry condition $\hat{T}(\varphi^{\circ})=(\hat{T}\varphi)^{\circ}$ (see Remark \ref{comparisonMonTop}), in particular the important extensions first considereded by Seal \cite{Seal}. In fact, in the
following theorem we give a context in which $T_{\xi}$ is described as {\em minimal} among extension families inducing $\xi$.
\end{itemize}
\end{rem}
For a commutative quantale $\sV$ and a $\Set$-monad ${\mathbb T}=(T,m,e)$, continuing to use the notations $\zeta:T1\to 1$ and $\phi^1:1\rto X\times Y$ whenever $\phi:X\rto Y$ in $\sV\text{-}{\bf Rel}$, let us call an extension family $\hat{T}=(\hat{T}\phi)_{\phi}$ (see Section 6) {\em admissible} if, for all $\phi$,
$$(\hat{T}\phi)^1\geq({\rm can}_{X,Y})_{\circ}\circ\hat{T}(\phi^1)\circ\zeta^{\circ},$$
and {\em algebraic}, if $``\geq"$ may always be replaced by ``="; here ${\rm can}_{X,Y}: T(X\times Y)\to
TX\times TY$ is the canonical map. (Note that this definition of algebraicity is just an element-free rendering of the
definition given above in a narrower context.) Denoting by
 $(\bbT,\sV)\text{-}\rm{EXT}_{\rm adm}$ the conglomerate of all admissible, left-op-whiskering and monotone extension families of $T$ (see Proposition \ref{corr}), one has a monotone map

 $$ \Xi:(\bbT,\sV)\text{-}{\rm{EXT}}_{{\rm adm}}\to\{\xi\in{\bf Set}(T\sV,\sV)\;|\;\xi \text{ monotone}\},\;\hat{T}\mapsto \zeta_!\cdot\overleftarrow{\hat{T}\varepsilon_1}=\overleftarrow{\hat{T}\varepsilon_1\circ\zeta^{\circ}},$$
 with monotonicity of arbitrary maps $T\sV\to\sV$ to be understood as in condition 3 above,  and with their order given pointwise as in $\sV$. The following Theorem shows that this map is an order embedding and has a right adjoint, given by
 $$\xi\mapsto T_{\xi}, \text{ with }(T_{\xi}\phi)^1=({\rm can}_{X,Y})_{\circ}\circ(T|\phi|)^{\circ}\circ\xi^{\circ}\circ\varepsilon_1$$
and $|\phi|=\overleftarrow{\phi^1}:X\times Y\to\sV$ as used in $(\dagger)$; in fact, the formula above is just an element-free rendering of the formula $(\dagger)$ of Theorem \ref{comparison theorem}.

 \begin{thm}\label{minimal}
 Let $T:\Set\to\Set$ satisfy BC, $\sV$ be a commutative quantale and the map $\xi:T\sV\to \sV$ be monotone. Then $T_{\xi}$ is the least of all admissible, left-op-whiskering and monotone extension families $\hat{T}$ with $\Xi(\hat{T})=\xi.$
 \end{thm}

 \begin{proof}
 First we verify that $T_{\xi}$ is left-op-whiskering, so that it satisfies condition (0) of Proposition \ref{corr}. Indeed, for $\phi:X\rto Y$ and $h:Z\to Y$, with $|h^{\circ}\circ \phi|=|\phi|\cdot(1_X\times h)$ one obtains
 $$(T_{\xi}(h^{\circ}\circ\phi))^1=({\rm can}_{X,Z})_{\circ}\circ (T|h^{\circ}\circ\phi|)^{\circ}\circ\xi^{\circ}\circ\varepsilon_1=({\rm can}_{X,Z})_{\circ}\circ(T(1\times h))^{\circ}\circ(T|\phi|)^{\circ}\circ\xi_{\circ}\circ\varepsilon_1.$$
 Since the satisfaction of BC by $T$ forces
 $$({\rm can}_{X,Z})_{\circ}\circ (T(1\times h))^{\circ}=(1_{TX}\times Th)^{\circ}\circ({\rm can}_{X,Y})_{\circ}$$
 (see Proposition 1.4.3 of \cite{ClementinoTholen}), the previous identity gives
 $(T_{\xi}(h^{\circ}\circ\phi))^1=((Th)^{\circ}\circ T_{\xi}\phi)^1$, as desired.

 For admissibility of $T_{\xi}$, first an easy inspection shows $\xi^{\circ}\circ\varepsilon_1=T_{\xi}\varepsilon_1\circ\zeta^{\circ}$.
 Since $T_{\xi}$ is left-op-whiskering, this identity and $\phi^1=|\phi|^{\circ}\circ\varepsilon_1$ in fact confirm even its algebraicity:
$$(T_{\xi}\phi)^1=({\rm can}_{X,Y})_{\circ}\circ(T|\phi|)^{\circ}\circ T_{\xi}\varepsilon_1\circ\zeta^{\circ}=
({\rm can}_{X,Y})_{\circ}\circ T_{\xi}(|\phi|^{\circ}\circ\varepsilon_1)\circ\zeta^{\circ}=
({\rm can}_{X,Y})_{\circ}\circ T_{\xi}(\phi^1)\circ\zeta^{\circ}.$$

 For an arbitrary admissible, left-op-whiskering and monotone $\hat{T}$ with $\Xi(\hat{T})=\xi$, we first use the left-op-whiskering property to obtain
 $\hat{T}(\phi^1)=(T|\phi|)^{\circ}\circ \hat{T}\varepsilon_1$ and then
 $$\overleftarrow{\hat{T}(\phi^1)\circ\zeta^{\circ}}=
 \zeta_!\cdot\overleftarrow{\hat{T}(\phi^1)}=\zeta_!\cdot \overleftarrow{\hat{T}\varepsilon_1}\cdot T|\phi|
 =\xi\cdot T|\phi|.$$
 Consequently, the admissibility of $\hat{T}$ gives the desired inequality
 $$(T_{\xi}\phi)^1=({\rm can}_{X,Y})_{\circ}\circ(T|\phi|)^{\circ}\circ\xi^{\circ}\circ\varepsilon_1
 =({\rm can}_{X,Y})_{\circ}\circ (\overleftarrow{\hat{T}(\phi^1)\circ\zeta^{\circ}})^{\circ} \circ\varepsilon_1
 =({\rm can}_{X,Y})_{\circ}\circ\hat{T}(\phi^1)\circ\zeta^{\circ}\leq(\hat{T}\phi)^1,$$
 which confirms the minimality of $T_{\xi}$.
  \end{proof}


\end{document}